\documentclass[a4paper, 11pt]{article}
\usepackage{amsmath,amssymb,esint,amscd,xspace,fancyhdr,color,authblk,srcltx,fontenc,bbm}
\usepackage{hyperref}
\usepackage{cite}

\newtheorem{theorem}{Theorem}[section]

\newtheorem{definition}[theorem]{Definition}
\newtheorem{lemma}[theorem]{Lemma}
\newtheorem{proposition}[theorem]{Proposition}

\newenvironment{proof}[1][Proof]{\textbf{#1.} }{\hfill\rule{0.5em}{0.5em}}
{\catcode`\@=11\global\let\AddToReset=\@addtoreset
\AddToReset{equation}{section}

\newcommand*{\bigchi}{\mbox{\Large$\chi$}}

\AddToReset{theorem}{section}

\title{Global Lorentz estimates for non-uniformly nonlinear elliptic equations via fractional maximal operators}
\author{Thanh-Nhan Nguyen\thanks{Department of Mathematics, Ho Chi Minh City University of Education, Ho Chi Minh City, Vietnam; \texttt{nhannt@hcmue.edu.vn}}, Minh-Phuong Tran\footnote{Corresponding author.}\thanks{Applied Analysis Research Group, Faculty of Mathematics and Statistics, Ton Duc Thang University, Ho Chi Minh City, Vietnam; \texttt{tranminhphuong@tdtu.edu.vn}}}

\date{\today}

\begin{document}
 
\maketitle
\begin{abstract}
This paper is a contribution to the study of regularity theory for  nonlinear elliptic equations. The aim of this paper is to establish some global estimates for non-uniformly elliptic in divergence form as follows
\begin{align*}
-\mathrm{div}(|\nabla u|^{p-2}\nabla u + a(x)|\nabla u|^{q-2}\nabla u) = - \mathrm{div}(|\mathbf{F}|^{p-2}\mathbf{F} + a(x)|\mathbf{F}|^{q-2}\mathbf{F}),
\end{align*}
that arises from double phase functional problems. In particular, the main results provide the regularity estimates for the distributional solutions in terms of maximal and fractional maximal operators. This work extends that of \cite{CoMin2016,Byun2017Cava} by dealing with the global estimates in Lorentz spaces. This work also extends our recent result in \cite{PNJDE}, which is devoted to the new estimates of divergence elliptic equations using cut-off fractional maximal operators. For future research, the approach developed in this paper allows to attain global estimates of distributional solutions to non-uniformly nonlinear elliptic equations in the framework of other spaces.

\medskip

\medskip

\medskip

\noindent 

\medskip


\end{abstract}   
                  
\section{Introduction and statement of main results}\label{sec:intro}
In mathematical analysis, the calculus of variations concerned with minimizing (or maximizing) energy functionals, where one wishes to find the minimum (or maximum) of a certain class of functions. Solutions to minimization (or maximization) problems in the calculus of variations lead to partial differential equations, in which studying a minimizer of a functional towards to the solution of Euler-Lagrange equation. Therefore, Calculus of Variations, as a door be opened wide for applied mathematics and theoretical physics, has applications in depth coverage a myriad of valuable problems from various fields of sciences: computer science, engineering, economics, biology, etc. In recent years, many researchers were inspired by the subject of the existence local minimizers, regularity properties of minimizers of energies, their extremality properties, and many others in different scenarios. 

In~\cite{Marcellini1989,Marcellini1991,Marcellini1996}, P. Marcellini was interested in the regularity properties of minimizers of a class of integral energy functionals
\begin{align}
\label{eq:classical_energy}
\mathcal{P}(\omega,\Omega) := \int_\Omega{f(x,\nabla\omega(x))dx},
\end{align}
which is defined for $\omega \in W^{1,1}(\Omega)$. Here, domain $\Omega$ is an open bounded domain in $\mathbb{R}^n$, $n \ge 2$, and the integrand $f := f(x,\xi): \Omega \times \mathbb{R}^n \to \mathbb{R}$ is strictly convex function with respect to $\xi$ and satisfies the following \emph{unbalanced growth}
\begin{align}
\label{eq:pq_growth}
C_1 |\xi|^p \le |f(x,\xi)| \le C_2 (1+ |\xi|^q), \quad \forall (x,\xi) \in \Omega \times \mathbb{R},
\end{align}
where $C_1\le C_2$ are positive constants and $1< p \le q$. The unbalanced condition~\eqref{eq:pq_growth} is also called the \emph{$(p,q)$-growth condition} or \emph{non-standard growth condition}, and it remarks that this condition is necessary for the existence of solution to~\eqref{eq:classical_energy} in $W^{1,p}(\Omega)$. The study of problem~\eqref{eq:classical_energy} has its origin in physics, such as nonlinear elasticity, fluidynamics and homogenization (see again~\cite{Marcellini1989,Marcellini1991,Marcellini1996} and~\cite{Fuchs} with further references therein). 

The study of Calder\'on-Zygmund and regularity theory of minimizers to \emph{double phase} variational problems has been receiving increased attention in recent years. For the time being, study of regularity properties of minimisers of \eqref{eq:classical_energy} has attracted much attention from many researchers, see~\cite{Esposito1999,Esposito1999JDE,Esposito2002,Esposito2004,Carozza2011}, to which we refer the interested readers. Furthermore, there have been various significant models with $(p,q)$ growth improved the integrability of minimizers, even for standard case $p=q$, we can indicate here the monographs~\cite{Manfredi1986,Manfredi1988,DM1993,Zhikov1997,CF1999,CRR2018,
CoMin215,BaCoMin2016,CoMin218,BaCoMin2018,Filippis2018}, that have been treated so far.

The concept of \emph{double phase} appeared whenever the integrand $f$ in problem~\eqref{eq:classical_energy} is mixed up two different types of degenerate elliptic phases, according to the positivity of the modulating coefficient $a(\cdot)$. Here, we include a more precise description of the double phase problem, that is the model of minimizing variational integrals
\begin{align}\label{eq:double_phase}
W^{1,1}(w) \ni w \mapsto \mathcal{P}_{p,q}(w,\Omega):= \int_\Omega{\left(|\nabla w|^p + a(x)|\nabla w|^q \right)dx},
\end{align}
together with initial assumptions are required that
\begin{align}\label{eq:apq}
\begin{split}
\begin{cases}
&1<p \le q <n; \\
&0 \le a(\cdot) \in C^{0,\alpha}, \quad 0<\alpha \le 1 ;   
\end{cases}
\end{split}
\end{align}
stems from physical phenomena and very general conditions to build regularity theory for minimizers.  The functional $\mathcal{P}_{p,q}(w,\Omega)$ was first considered by V.V. Zhikov in~\cite{Zhikov1986, Zhikov1995, Zhikov1997} to provide models for strongly anisotropic materials in the context of homogenisation, and afterwards, this functional arises in different contexts from mathematical physics (see for example~\cite{Ball1982,Ruzicka,Benci2000}). Moreover, in model~\eqref{eq:double_phase}, the function $a:\Omega \to [0,\infty)$ in physical sense, dictates the object's geometry of the composite material formed by mixing of two distinct materials together, with power hardening $p$ and $q$, respectively. The presence of $a(\cdot)$ in energy functional $\mathcal{P}_{p,q}$ dictates energy growth, that means on the set $\{x \in \Omega: a(x)=0\}$ the functional $\mathcal{P}_{p,q}$ has $p$-growth in the gradient and otherwise, the growth is of order $q$. It is noteworthy that the function $a(\cdot)$ interacts with the ratio $q/p$ and due to the appearance of \emph{Lavrentiev's phenomenon} arising in mathematical physics, where the minimisers may be discontinuous (see~\cite{Lavrentiev, Zhikov1995} and~\cite{Esposito2004,FMM2004}), it is natural to construct regularity theory for local minimisers under certain assumptions required for $a(\cdot), p, q$ in~\eqref{eq:apq} and 
\begin{align}
\label{eq:pq_strict}
\frac{q}{p} < 1 +\frac{\alpha}{n}.
\end{align}

Furthermore, in the theory of regularity, problem~\eqref{eq:double_phase} leads to solving the associated Euler–Lagrange equation of type
\begin{align}
\label{eq:euler-lagrange}
-\mathrm{div}(p|\nabla u|^{p-2}Du + qa(x)|\nabla u|^{q-2}\nabla u) = 0,
\end{align}
and this is called the non-uniformly elliptic equation, involved the \emph{non-uniformly elliptic operator} 
\begin{align}
\label{eq:A}
\mathcal{A}(x,\xi):= p|\xi|^{p-2}\xi + qa(x)|\xi|^{q-2}\xi,
\end{align}
for $x \in \Omega$ and $\xi \in \mathbb{R}^n$, in which $p, q$ and $a$ satisfy \eqref{eq:apq}. Different from the uniformly elliptic operators, the \emph{ellipticity ratio} of this operator (ratio between the highest and the lowest eigenvalue of $\partial\mathcal{A}(x,\xi)$) might be unbounded (see~\cite{CoMin215,CoMin2016,FM2019JGA} for more detailed explanations and examples). For the case $a \equiv 0$, problem~\eqref{eq:euler-lagrange} becomes the classical $p$-Laplacian equation $-\Delta_p u=0$, from which several regularity results have been obtained in earlier works~\cite{Ural1968,Uhlenbeck1977,Evans1982,DiBenedetto1983,Iwaniec1983}. For the general case,  there have been also numerous monographs concerning regularity results of non-uniformly elliptic operators, such as~\cite{Ural1970,Simon76,Ural1984} and other research studies. 

Let us briefly review some recent important results  on the double phase variational problem~\eqref{eq:double_phase}, as well as some essential problems generated from previous research. We firstly refer to the works by E. Esposito {\it et al.} in~\cite{Esposito1999,Esposito2002,Esposito2004}, where the higher gradient integrability of solutions to variational problems with $(p,q)$-growth was carried. Recently, M. Colombo and G. Mingione in their early papers~\cite{CoMin215,CoMin218} that studied the H\"older estimates for gradient of solutions and the sharp regularity results for a class of integral functionals in~\eqref{eq:double_phase}, respectively. In particular, in~\cite{CoMin215}, the local H\"older gradient continuity of minima was obtained under the main assumption $\frac{q}{p}<1+\frac{\alpha}{n}$.  Otherwise, the higher integrability of minimisers was also proved by the same group of authors in~\cite{CoMin218} under another sharp assumption $q \le p+\alpha$ (dimension free bound). For year, P. Baroni {\it et al.} in~\cite{BaCoMin2015} completed the study of bounded minimisers by the proof of borderline case, where $\frac{q}{p} = 1+\frac{\alpha}{n}$. Since then, the regularity properties for some class of functionals with $(p,q)$ growth conditions (the borderline case was also included) or non-standard growth conditions, were also the research topics of~\cite{BaCoMin2015,BaCoMin2016,BaCoMin2018,Filippis2018}.  Besides, recently there has been growing interest in methods to prove regularity properties for more general classes of non-uniformly elliptic operators, especially in geometric viewpoint. For instance, the regularity theory for a class of non-uniformly variational problems in the case when both minimizers and competitors take values into a manifold was the first model proposed by C. De Filippis {\it et al.} in~\cite{Filippis2018,FM2019JGA,FP2019} and it allows the study of non-uniformly elliptic problems in the context of geometry to be carried.

1.1. \textbf{Some known results and our motivation.} In this paper, we consider the following non-uniformly elliptic equation of the form
\begin{align}\label{eq:diveq}
-\mathrm{div}(p|\nabla u|^{p-2}\nabla u + qa(x)|\nabla u|^{q-2}\nabla u) = - \mathrm{div}(|\mathbf{F}|^{p-2}\mathbf{F} + a(x)|\mathbf{F}|^{q-2}\mathbf{F}),
\end{align}
that occurs as the Euler-Lagrange equation of energy functional
\begin{align}\nonumber 
W^{1,1}(\Omega) \ni w \mapsto \mathcal{P}_{p,q}(w,\Omega) - \int_\Omega{\langle |\mathbf{F}|^{p-2}\mathbf{F} + a(x)|\mathbf{F}|^{q-2}\mathbf{F}\rangle dx},
\end{align}
where $F:\Omega \to \mathbb{R}^n$ is a vector field, the given numbers $p,q$ and coefficient $a(\cdot)$ satisfy assumptions in~\eqref{eq:apq}. Let us roughly and briefly review some known regularity results pertaining to equation~\eqref{eq:diveq} or some non-uniformly elliptic equations of that form. 

From prior theories and significant research by M. Colombo and G. Mingione for variational integrals problems in~\cite{CoMin215,CoMin218}, an article by the same group of authors was first investigated the local version of Calder\'on-Zygmund for the solutions to~\eqref{eq:diveq} in~\cite{CoMin2016}, that resulted
\begin{align}\label{eq:CZloc}
|\mathbf{F}|^p + a(x)|\mathbf{F}|^q \in L^\gamma_{\mathrm{loc}}(\Omega) \Longrightarrow |\nabla u|^p + a(x)|\nabla u|^q \in L^\gamma_{\mathrm{loc}}(\Omega),
\end{align}
holds for every $\gamma \in [1,\infty)$, when $\frac{q}{p}<1+\frac{\alpha}{n}$. Moreover, authors also dealt with results under the imposed assumption $q \le p+\alpha$ that independent on the dimension $n$ and these interesting results were further extended to the vectorial case  in the same paper.

Particularly interesting results came to the attention of other researchers in the past years.  For instance,  S. Byun {\it et al.} in~\cite{Byun2017Cava} provided the extension of such results up to the boundary, that states the global $L^\gamma$ estimates in terms of Calder\'on-Zygmund with $\partial\Omega$ is $C^{1,\alpha^+}$ domain, for $\alpha^+ \in [\alpha,1]$. On the other hand, C. De Filippis {\it et al.} have driven the further development in \cite{FM2019} also claimed the validity of sharp relation in~\eqref{eq:CZloc} in the delicated borderline case $$\frac{q}{p}=1+\frac{\alpha}{n}.$$

In the context of a class of quasilinear elliptic equations involved \emph{uniformly elliptic operator} $\mathcal{A}$, there are numerous different approaches established for regularity results of $-\mathrm{div}\mathcal{A}(x,\nabla u) = \texttt{right hand side}$, where the `\texttt{right hand side}' can be given as the general functional datum $\mathbf{F}$; the divergence form $\mathrm{div}(|\mathbf{F}|^{p-2}\mathbf{F})$; or measure datum $\mu$, under various assumptions on domain, nonlinear operator $\mathcal{A}$ and Dirichlet boundary data (homogeneous or non-homogeneous). For instance, $L^q$ and $W^{1,q}$ estimates to the quasilinear elliptic equations of type $\mathrm{div}\mathcal{A}(x,\nabla u) = \mathrm{div}\mathbf{F}$ were discussed and addressed in the interesting series of papers~\cite{Caffarelli1998, MP11, Phuc2013, BW1, SSB2, SSB3, SSB4, SSB1, Tuoc2018, FT2018} and their related references. Otherwise, regularity of solutions to nonlinear elliptic equations of divergence form $\mathrm{div}\mathcal{A}(x,\nabla u) = \mathrm{div}(|\mathbf{F}|^{p-2}\mathbf{F})$ were also established in Lebesgue, Sobolev's spaces, see~\cite{BCDKS, BW1, CM2014, KZ}; and later extensively treated in more general functional spaces, such as Lorentz, Morrey, Lorentz-Morrey or even Orlicz spaces, etc, \cite{Phuc2015, MPT2018, PNJDE, PNCRM}, to which we refer the interested readers. 

To our knowledge, there has been several approaches devoted to the study of regularity estimates for nonlinear elliptic equations over the years. It is worth for us to mention the method based on the Vitali type covering lemma by S.S. Byun {\it et al.}~\cite{SSB1,SSB3,SSB4}; the method based on the boundedness properties of singular integral potentials by G. Mingione {\it et al.}~\cite{Duzamin2,55DuzaMing,Min2007}; or the method based on the good-$\lambda$ type bounds in~\cite{55QH2,55QH4,MPT2018,PNCCM,PNJDE,PNCRM}. The reader can get a more or less complete picture from the monographs~\cite{Min5,Mi2019} and the references cited therein. It can be seen that the presence of Vitali's covering technical lemma and the properties of reverse H\"older type inequality may be successfully applied in most of these research contributions. On the other hand, our approach in~\cite{PNJDE} using features from \emph{cut-off fractional maximal functions}, was addressed, that becomes a promising  technique to achieve the fractional maximal gradient estimates for solutions to nonlinear elliptic equations. 

Motivated by above interesting results and mathematical techniques developed for quasilinear elliptic equations, our approach in this paper is to establish the global Calder\'on-Zygmund type estimates for solutions to non-uniformly equations \eqref{eq:diveq} in Lorentz spaces, via strong maximal and fractional maximal functions. From ideas seem to come out of previous research findings, our work here is not only an extension of the results in~\cite{CoMin2016,FM2019} for the global regularity in Lorentz spaces, but also an improved result using the technique proposed in our previous work~\cite{PNJDE}, in which the cut-off fractional maximal functions included. Furthermore, the good-$\lambda$ technique plays a role in our proofs to obtain regularity results in the interior and up to the boundary of domain.

1.2. \textbf{Statement of main results.}
Our results in this paper are in fact proved for a more general quasilinear elliptic equation than that of~\eqref{eq:diveq}. To be more precise, we are interested in the following elliptic equations with the homogeneous Dirichlet boundary condition of the type
\begin{align}\tag{P}
\label{eq:main_double}
\begin{cases}
\mathrm{div}(\mathcal{A}(x,\nabla u)) & =  \ \mathrm{div}(\mathcal{B}(x,\mathbf{F})) \quad \text{in} \ \ \Omega, \\
\hspace{1.2cm} u & =  \ 0 \qquad \qquad \qquad \text{on} \ \ \partial \Omega,
\end{cases}
\end{align}
where $\Omega$ is an open bounded domain in $\mathbb{R}^n$ with $n \ge 2$, the datum $\mathbf{F}: \Omega \to \mathbb{R}^n$ is a vector field. The function $a: \Omega \to [0,\infty)$ and parameters $p,q$ satisfy the following main assumptions
\begin{align}\tag{A1}
\label{eq:cond1}
&0 \le a(\cdot) \in C^{0,\alpha}, \quad \alpha \in (0,1] ; \\ 
&1 < p < q \le \left(1+\displaystyle{\frac{\alpha }{n}}\right)p. \tag{A2}
\label{eq:cond2}
\end{align}
The functional operator $\mathcal{A}:\Omega \times \mathbb{R}^n \to \mathbb{R}^n$ defined in~\eqref{eq:A}, is measurable with respect to $x$, differentiable with respect to $y \neq 0$ and satisfies the following conditions 
\begin{align}\tag{A3}
\begin{cases}
|\mathcal{A}(x,y)| + |\partial \mathcal{A}(x,y)||y| \le L \left(|y|^{p-1}+a(x)|y|^{q-1} \right);\\
\nu \left(|y|^{p-2}+a(x)|y|^{q-2} \right)|z|^2 \le  \langle \partial \mathcal{A}(x,y)z,z \rangle; \\
\left|\mathcal{A}(x_1,y) - \mathcal{A}(x_2,y) \right| \le L |a(x_1) - a(x_2)||y|^{q-1},
\end{cases}
\label{eq:cond3}
\end{align}
whenever $y, z \in \mathbb{R}^n \setminus \{0\}$; $x,x_1,x_2 \in \Omega$ and $0<\nu \le L < +\infty$ are fixed constants. One notices that $\partial$~denotes the partial differentiation with respect to the gradient variable $y$. In the view of condition~\eqref{eq:cond3}$_2$, for $1<p<q$, it implies an additive monotonicity property of $\mathcal{A}$ as follows
\begin{align}
\label{eq:dk8}
\tilde{\nu}\left[(|y_1|^2+|y_2|^2)^{\frac{p-2}{2}} + a(x)(|y_1|^2+|y_2|^2)^{\frac{q-2}{2}} \right]|y_1-y_2|^2 \le \langle \mathcal{A}(x,y_1)-\mathcal{A}(x,y_2),y_1-y_2 \rangle,
\end{align}
where $\tilde{\nu}$ is another positive constant depending only on $n,p,q,\nu$. And for the particular case $2 \le p<q$, we can do to reduce it as
\begin{align}\label{eq:dk9}
\tilde{\nu} \left(|y_1-y_2|^p + a(x)|y_2-y_2|^q \right) \le \langle \mathcal{A}(x,y_1) - \mathcal{A}(x,y_2),y_1-y_2 \rangle.
\end{align}

In addition, the vector field $\mathcal{B}$ is a Carath\'eodory vector valued function (that is, $\mathcal{B}(.,y)$ is measurable on $\Omega$ for every $y$ in $\mathbb{R}^n$, and $\mathcal{B}(x,.)$ is continuous on $\mathbb{R}^n$ for almost every $x$ in $\Omega$) which satisfies the following growth condition
\begin{align}\label{eq:B-cond}
|\mathcal{B}(x,y)| \le L (|y|^{p-1} +a(x)|y|^{q-1}),
\end{align}
for every $x \in \Omega$ and $y \in \mathbb{R}^n$. It can be seen that the problem~\eqref{eq:main_double} under some assumptions~\eqref{eq:cond1}, \eqref{eq:cond2}, \eqref{eq:cond3} contains the model~\eqref{eq:diveq} as a special case. 

Before formulating our main results, let us introduce some important and relevant terminology. In the remainder of this paper, we shall denote by $\mathcal{H}$ the operator 
\begin{align}\nonumber 
\mathcal{H}(x,y) = |y|^p + a(x)|y|^q,
\end{align}
for every $x \in \Omega$ and $y \in \mathbb{R}^n$.  Under various assumptions in~\eqref{eq:cond1}, \eqref{eq:cond2} and~\eqref{eq:cond3}, there exists a certain set of parameters that will affect the constant dependence in our statements below and for the convenience of the reader, let us set the notation 
\begin{align*}
\texttt{data} \equiv \texttt{data}(n,p,q,\alpha,\nu,L,\|a\|_{L^\infty},[a]_\alpha,\|\mathcal{H}(\cdot,\nabla u)\|_{L^1}),
\end{align*}
to simplify the dependence on known data of the problem.

The purpose of this paper is twofold. On the one hand, it firstly attempts to extend results mentioned in~\cite{CoMin2016,Byun2017Cava} to classical Lorentz spaces, via maximal functional operators. Secondly, it aims at deriving the global fractional maximal estimate of solutions, more general types of regularity results in Lorentz spaces. Our main results are stated in the following theorems. The first result in Theorem~\ref{theo:regularityM} is the global Calder\'on-Zygmund type estimates for problem~\eqref{eq:main_double} in Lorentz spaces, via the classical maximal operators. The next result of this paper deals with a more general than that of Theorem~\ref{theo:regularityM} to fractional maximal operators, will be clarified in Theorem~\ref{theo:main-M-beta} below. This result is essentially just a re-statement of Theorem~\ref{theo:M_lambda} in more extensive version. Our proofs, presented in the Section~\ref{sec:proofs}, can be established using the corresponding good-$\lambda$ technique also stated and justified therein.
\begin{theorem}\label{theo:regularityM}
Let $\Omega$ be an open bounded domain in $\mathbb{R}^n$ such that $\partial \Omega$ is $C^{1,\alpha^+}$ domain for some $\alpha^+ \in [\alpha,1]$. Assume that $u \in W^{1,1}(\Omega)$ is a distributional solution to~\eqref{eq:main_double} with 
\begin{align*}
\mathcal{H}(x,\nabla u); \ \mathcal{H}(x,\mathbf{F}) \in L^1(\Omega),
\end{align*}
under main assumptions given in~\eqref{eq:cond1}, \eqref{eq:cond2} and~\eqref{eq:cond3}. 
Then for every $0<s<\infty$ and $0<t \le \infty$, there holds 
$$\mathbf{M}(\mathcal{H}(x,\mathbf{F})) \in L^{s,t}(\Omega) \Longrightarrow \mathbf{M}(\mathcal{H}(x,\nabla u)) \in L^{s,t}(\Omega)$$ 
with the following corresponding estimate 
\begin{align}\label{eq:regularityM0}
\|\mathbf{M}(\mathcal{H}(x,\nabla u))\|_{L^{s,t}(\Omega)}\leq C \|\mathbf{M}(\mathcal{H}(x,\mathbf{F}))\|_{L^{s,t}(\Omega)}.
\end{align}
Here, $C$ is the positive constant depending only on $\texttt{data},\Omega,s,t$.
\end{theorem}
\begin{theorem}\label{theo:main-M-beta}
Let  $\beta \in [0, n)$ and the equation~\eqref{eq:main_double} is set under conditions in~\eqref{eq:cond1}, \eqref{eq:cond2} and~\eqref{eq:cond3}, where $\Omega$ is an open bounded domain in $\mathbb{R}^n$ such that $\partial \Omega \in C^{1,\alpha^+}$ for some $\alpha^+ \in [\alpha,1]$. Assume that $u \in W^{1,1}(\Omega)$ is a distribution solution to~\eqref{eq:main_double} with given data $\mathbf{F}$ satisfying 
$$\mathcal{H}(x,\nabla u), \ \mathcal{H}(x,\mathbf{F}) \in L^1(\Omega).$$ 
Then, for every $s \in (0,\infty)$ and $0 < t \le \infty$,  there exists a constant~$C = C(\texttt{data},\Omega,s,t,\beta)>0$ such that 
\begin{align}\label{eq:main-M-beta}
\|\mathbf{M}_{\beta}(\mathcal{H}(x,\nabla u))\|_{L^{s,t}(\Omega)}\leq C \|\mathbf{M}_{\beta}(\mathcal{H}(x,\mathbf{F}))\|_{L^{s,t}(\Omega)}.
\end{align}
\end{theorem}

The paper is organized as follows. In the next section \ref{sec:pre} we point out some notation and definitions that will be used throughout the paper. Section \ref{sec:comparison_sec} is devoted to state and prove comparison results via some important lemmas of local estimates and their boundary versions. The last section leads to the proofs of main theorems, which in particular gives two separable regularity results of maximal and fractional maximal estimates of solutions in Lorentz spaces. 

\section{Notation and Preliminaries}
\label{sec:pre}
Let us first introduce in this section some notations and initial definitions and properties that will be used in the rest of the paper.

Throughout the study, the domain $\Omega \subset \mathbb{R}^n$, $n \ge 2$ is assumed to be an open bounded domain and the denotation $B_R(x)$ stands for an open ball in $\mathbb{R}^n$ with radius $R>0$ and center $x \in \mathbb{R}^n$; that is the set $\{y \in \mathbb{R}^n: |y-x|<R\}$. And in what follows, let us denote $\displaystyle{\fint_{B}{f(x)dx}}$, the average value of a function $f \in L^1_{\mathrm{loc}}(\mathbb{R}^n)$ over an open bounded set $B \subset \mathbb{R}^n$ as
\begin{align*}
\fint_B{f(x)dx} = \frac{1}{\mathcal{L}^n(B)}\int_B{f(x)dx}.
\end{align*}
Here the notation $\mathcal{L}^n(E)$ stands for the $n$-dimensional Lebesgue measure of a measurable set $E \subset \mathbb{R}^n$. In addition, the denotation $\mathrm{diam}(S)$ is the diameter of a set $S \subset \Omega$ defined as:
\begin{align*}
\mathrm{diam}(S) = \sup \{d(x,y) \  : \ x,y \in S\}.
\end{align*}
Hereafter, for the sake of convenience, with some abuse of notation, the set $\{x \in \Omega: |g(x)| > \Lambda\}$ is still denoted by $\{|g|>\Lambda\}$. Otherwise, for the function $a$ mentioned in \eqref{eq:cond1}, we denote
\begin{align*}
[a]_{\alpha;S} := \sup_{x,y \in S; \, x \neq y}{\frac{|a(x)-a(y)|}{|x-y|^\alpha}},
\end{align*}
for any nonempty set $S \in \Omega$. For simplicity's sake, in the present paper we term $[a]_\alpha$ for $[a]_{\alpha;\Omega}$ and $\|a\|_{L^\infty}$ for $\|a\|_{L^\infty(\Omega)}$.

Further, in order not to confuse the issue we shall denote by $C$ a universal constant depending only on some prescribed parameters, that may be different from line to line. The dependencies of constant $C$ will be emphasized between parentheses (\texttt{data} may be included), which greater than or equal to 1. 

\begin{definition}[Distributional solution]\label{def:weak_sol}
A function $u \in W^{1,1}_0(\Omega)$ is a distributional solution to \eqref{eq:main_double} under assumptions~\eqref{eq:cond1}, \eqref{eq:cond2} and~\eqref{eq:cond3} if it satisfies
\begin{align*}
\int_\Omega{\langle \mathcal{A}(x,\nabla u),D\varphi \rangle dx} = \int_\Omega{\langle \mathcal{B}(x,\mathbf{F}),D\varphi \rangle dx},
\end{align*}
for every $\varphi \in C_0^\infty(\Omega)$.
\end{definition}
In~\cite[Proposition 3.5]{Byun2017Cava}, an important result concerned to distributional solution of~\eqref{eq:main_double} was addressed. For the reader's convenience, we re-state the result in the following lemma.
\begin{lemma}\label{lem:solu}
Let $u \in W^{1,1}(\Omega)$ be a distribution solution to~\eqref{eq:main_double} under conditions in~\eqref{eq:cond1}, \eqref{eq:cond2} and~\eqref{eq:cond3} with given data $\mathbf{F}$ satisfying $\mathcal{H}(x,\nabla u), \ \mathcal{H}(x,\mathbf{F}) \in L^1(\Omega)$. Then the following variational formula 
\begin{align}\label{eq:solu}
\int_\Omega{\langle \mathcal{A}(x,\nabla u),\nabla \varphi \rangle dx} = \int_\Omega{\langle \mathcal{B}(x,\mathbf{F}), \nabla \varphi \rangle dx}
\end{align}
holds for every test function $\varphi \in W_0^{1,1}(\Omega)$ such that $\mathcal{H}(x,\nabla \varphi) \in L^1(\Omega)$.
\end{lemma}
We also recall here the definition of Lorentz spaces, are the ones that more general than Lebesgue spaces with two-parameter scale. The definition of Lorentz spaces $L^{s,t}(\Omega)$ can be found in~\cite{55Gra,Maly}.

\begin{definition}[Lorentz spaces]\label{def:Lorentz}
Lorentz space $L^{s,t}(\Omega)$ for $0<s<\infty$ and $
0 <t \le \infty$ is defined by that for all Lebesgue measurable function $f$ on $\Omega$, there holds
\begin{align*}
\|f\|_{L^{s,t}(\Omega)} := \left[ s \int_0^\infty{ \lambda^t\mathcal{L}^n \left( \{y \in \Omega: |f(y)|>\lambda\} \right)^{\frac{t}{s}} \frac{d\lambda}{\lambda}} \right]^{\frac{1}{t}} < \infty,
\end{align*}
if $s<\infty$ and otherwise
\begin{align*}
\|f\|_{L^{s,\infty}(\Omega)} := \sup_{\lambda>0}{\lambda \mathcal{L}^n\left(\{y \in \Omega:|f(y)|>\lambda\}\right)^{\frac{1}{s}}},
\end{align*}
this is defined as Marcinkiewicz spaces.
\end{definition}
It is worthy to remark that when $s=t$ the Lorentz space $L^{s,s}(\Omega)$ is nothing but a Lebesgue space $L^s(\Omega)$. In particular, it is well known that for some $0<r \le s \le t \le \infty$, there holds
$$L^t(\Omega)  \subset L^{s,r}(\Omega) \subset  L^{s}(\Omega) \subset L^{s,t}(\Omega) \subset L^r(\Omega).$$

Now, we turn our attention to the definition of fractional maximal function in the spirit of~\cite{K1997, KS2003}, that is a very important tool for the proofs of our results in the sequel. 

\begin{definition}[Fractional maximal function]\label{def:Malpha}
Let $0 \le \beta \le n$, the fractional  maximal function $\mathbf{M}_\beta$ of a locally integrable function $f: \mathbb{R}^n \rightarrow \mathbb{R}$ is defined by:
\begin{align}\label{eq:Malpha}
\mathbf{M}_\beta f(x) = \sup_{\rho>0}{\rho^\beta \fint_{B_\rho(x)}{|f(y)|dy}}, \quad x \in \mathbb{R}^n.
\end{align}
When $\beta=0$, it coincides with the Hardy-Littlewood maximal function, $\mathbf{M}_0f = \mathbf{M}f$, defined by:
\begin{align}\label{eq:M0}
\mathbf{M}f(x) = \sup_{\rho>0}{\fint_{B_{\rho}(x)}|f(y)|dy}, \quad x \in \mathbb{R}^n,
\end{align}
for a given locally integrable function $f$ in $\mathbb{R}^n$.
\end{definition}

Recently, the maximal function $\mathbf{M}$ is a classical tool that was studied in functional spaces, Calculus of Variations, partial differential equations and so on. In an interesting work, the theorem by Hardy, Littlewood and Wiener asserted a fundamental result of maximal operator. That is the boundedness of maximal function $\mathbf{M}$ on $L^p(\mathbb{R}^n)$ when $1< p \le \infty$ as follows
\begin{align*}
\|\mathbf{M}f\|_{L^p(\mathbb{R}^n)} \le C\|f\|_{L^p(\mathbb{R}^n)}, \quad \forall f \in L^p(\mathbb{R}^n),
\end{align*}
where $C$ is a constant depending only on $n$ and $p$. Moreover, $\mathbf{M}$ is also said to be \emph{weak-type (1,1)} proved in a theorem by Hardy-Littlewood-Wiener, that means for all $\lambda>0$ and $f \in L^1(\mathbb{R}^n)$, it holds that
\begin{align*}
\mathcal{L}^n\left(\{\mathbf{M}f>\lambda\}\right) \le \frac{C(n)}{\lambda} \|f\|_{L^1(\mathbb{R}^n)}.
\end{align*}
Besides that, in this paper we also review and collect some well-known properties of maximal and fractional maximal operators,  whose proofs can be further given in many references as~\cite{AH, Gra97, 55Gra}.

\begin{lemma}\label{lem:boundM}
The operator $\mathbf{M}$ is bounded from $L^s(\mathbb{R}^n)$ to $L^{s,\infty}(\mathbb{R}^n)$ for $s \ge 1$, this means
\begin{align*}
\mathcal{L}^n\left(\{\mathbf{M}f>\lambda\}\right) \le \frac{C}{\lambda^s}\int_{\mathbb{R}^n}{|f(x)|^s dx}, \quad \mbox{ for all } \lambda>0.
\end{align*}
\end{lemma}
\begin{lemma}\label{lem:boundMlorentz}
In~\cite{55Gra}, it allows us to present a boundedness property of maximal function $\mathbf{M}$ in the Lorentz space $L^{s,t}(\mathbb{R}^n)$, for $s>1$ as follows:
\begin{align*}
\|\mathbf{M}f\|_{L^{s,t}(\mathbb{R}^n)} \le C \|f\|_{L^{s,t}(\mathbb{R}^n)}.
\end{align*}
\end{lemma}
The concept of \emph{cut-off fractional maximal function} was first mentioned in~\cite{PNJDE} in order to prove the gradient estimates of quasilinear elliptic equations in divergence form, and it becomes an effective tool and strategy to attain our desired results. In this study, it allows us to review some basic definitions and crucial properties associated to these operators, that will be proved if necessary.

\begin{definition}[Cut-off fractional maximal functions]\label{def:cut-off}
Let $\varrho>0$ and $0\le \beta \le n$, we define some additional cut-off maximal functions of a locally integrable function $f$ corresponding to the maximal function $\mathbf{M}f$ in \eqref{eq:M0} as follows
\begin{align*}
{\mathbf{M}}^{\varrho}f(x)  &= \sup_{0<\rho<\varrho} \fint_{B_\rho(x)}f(y)dy; \ \ {\mathbf{T}}^{\varrho}f(x) = \sup_{\rho \ge \varrho}\fint_{B_\rho(x)}f(y)dy,
\end{align*}
and corresponding to $\mathbf{M}_{\beta}f$ in~\eqref{eq:Malpha} as
\begin{align*}
{\mathbf{M}}^{\varrho}_{\beta}f(x)  &= \sup_{0<\rho<\varrho} \rho^{\beta} \fint_{B_\rho(x)}f(y)dy; \ \ {\mathbf{T}}^{\varrho}_{\beta}f(x) = \sup_{\rho \ge \varrho} \rho^{\beta}\fint_{B_\rho(x)}f(y)dy.
\end{align*}
\end{definition}
We remark here that if $\beta = 0$ then $\mathbf{M}^r_{\beta}f = \mathbf{M}^rf$ and $\mathbf{T}^r_{\beta}f = \mathbf{T}^r f$, for all $f \in L^1_{\mathrm{loc}}(\mathbb{R}^n)$. The following lemma can be inferred from from their definitions.

\begin{lemma}\label{lem:bound-M-beta}
Let $s \ge 1$ and $\beta \in \left[0,\frac{n}{s}\right)$. There exists $C=C(n,\beta)>0$ such that for any $f \in L^s(\mathbb{R}^n)$ there holds
\begin{align*}
\mathcal{L}^n\left(\left\{x \in \mathbb{R}^n: \ \mathbf{M}_{\beta}f(x)>\lambda\right\}\right) \le C \left(\frac{1}{\lambda^{s}}\int_{\mathbb{R}^n}|f(y)|^sdy\right)^{\frac{n}{n-\beta s}},
\end{align*}
for all $\lambda>0$.
\end{lemma}
\begin{proof}
For any $x \in \mathbb{R}^n$, using definition of fractional maximal function $\mathbf{M}_{\beta}$ and H{\"o}lder's inequality we have
\begin{align*}
\left[\mathbf{M}_{\beta} f(x)\right]^s &= \left(\sup_{\rho>0} \rho^{\beta} \fint_{B_{\rho}(x)}|f(y)|dy\right)^s \\ 
& \le \sup_{\rho>0} \rho^{\beta s} \fint_{B_{\rho}(x)}|f(y)|^sdy = \mathbf{M}_{\beta s} (|f|^s)(x).
\end{align*}
It guarantees that for all $\lambda>0$ there holds
\begin{equation*}
\mathcal{L}^n \left(\left\{x \in \mathbb{R}^n: \ \mathbf{M}_{\beta} f(x) > \lambda \right\}\right) \le \displaystyle{\mathcal{L}^n\left(\left\{x \in \mathbb{R}^n: \ \mathbf{M}_{\beta s} (|f|^s)(x)> \lambda^s\right\}\right)}.
\end{equation*}
Moreover, by \cite[Lemma 2.3]{PNJDE} we know that
\begin{equation*}
\displaystyle{\mathcal{L}^n\left(\left\{x \in \mathbb{R}^n: \ \mathbf{M}_{\beta s} (|f|^s)(x)> \lambda^s\right\}\right)} \leq C(n,\beta)\displaystyle{\left(\frac{1}{\lambda^s} {\displaystyle{\int_{\mathbb{R}^n}|f(y)|^sdy}}\right)^{\frac{n}{n-\beta s}}},
\end{equation*}
which completes the proof.
\end{proof}

\begin{lemma}\label{lem:MrMr}
For any $\varrho>0$ and $\beta \in [0,n)$, there holds
\begin{align}\nonumber
{\mathbf{M}}\mathbf{M}_{\beta}f(x) &\le \max \left\{ \mathbf{M}^{2\varrho}_{\beta}f(x); \ \mathbf{M}^{\varrho} \mathbf{T}_{\beta}^{\varrho}f(x) ; \  \mathbf{T}^{\varrho} \mathbf{M}_{\beta}f(x) \right\},
\end{align}
for all $x \in \mathbb{R}^n$ and $f \in L^1_{\mathrm{loc}}(\mathbb{R}^n)$.
\end{lemma}
\begin{proof}
The conclusion of this lemma can be directly obtained by applying~\cite[Lemma 3.1 and Lemma 3.3]{PNJDE}.
\end{proof}

In this section, the following statement of Young's inequality will be necessary for later use. This inequality is standard, but we also offer a short proof here for the convenience of the reader.
\begin{lemma}[H{\"o}lder and Young's inequality] \label{lem:HY}
Let $\varphi, \psi \in L^r(\Omega)$, there holds
\begin{align}\label{HY-inq}
\int_{\Omega} |\varphi(x)|^{r-s}|\psi(x)|^sdx \le \varepsilon \int_{\Omega} |\varphi(x)|^r dx +  \varepsilon^{1- \frac{r}{s}} \int_{\Omega} |\psi(x)|^rdx,
\end{align}
for all $\varepsilon>0$ and $0 <s < r$.
\end{lemma}
\begin{proof}
Thanks to H{\"o}lder and Young's inequality, there holds
\begin{align*}
\int_{\Omega} |\varphi(x)|^{r-s}|\psi(x)|^sdx & \le \left(\int_{\Omega} |\varphi(x)|^{r} dx\right)^{\frac{r-s}{r}} \left(\int_{\Omega} |\psi(x)|^{r} dx\right)^{\frac{s}{r}} \\
& = \left(\varepsilon \int_{\Omega} |\varphi(x)|^{r} dx\right)^{\frac{r-s}{r}} \left(\varepsilon^{-\frac{r-s}{s}}\int_{\Omega} |\psi(x)|^{r} dx\right)^{\frac{s}{r}} \\
& \le \frac{r-s}{r} \varepsilon \int_{\Omega} |\varphi(x)|^{r} dx + \frac{s}{r} \varepsilon^{-\frac{r-s}{s}}  \int_{\Omega} |\psi(x)|^{r}dx,
\end{align*}
which deduce to~\eqref{HY-inq} with notice that $0<s/r<1$.
\end{proof}

\section{Comparison results}
\label{sec:comparison_sec}
In this section, we discuss some local interior and boundary comparison estimates for solution in \eqref{eq:main_double}, that imply the validity of global estimates in our development later.

\begin{proposition}\label{prop1}
Let $u \in W^{1,1}_0(\Omega)$ be a distribution solution to~\eqref{eq:main_double} under conditions in~\eqref{eq:cond1}, \eqref{eq:cond2} and~\eqref{eq:cond3} with given data $\mathbf{F}$ satisfying $\mathcal{H}(x,\nabla u), \ \mathcal{H}(x,\mathbf{F}) \in L^1(\Omega)$. Then we have
\begin{equation}\label{eq:prop1}
\int_{\Omega} \mathcal{H}(x,\nabla u) dx \le C \int_{\Omega} \mathcal{H}(x,\mathbf{F}) dx.
\end{equation}
Here, it remarks that the constant $C$ depends only on \texttt{data}.
\end{proposition}
\begin{proof}
Since $u \in W^{1,1}_0(\Omega)$ and $\mathcal{H}(x,\nabla u) \in L^1(\Omega)$, applying Lemma~\ref{lem:solu} we may choose $u$ as a test function in variational formula~\eqref{eq:solu} to obtain
\begin{align*}
\int_\Omega{\langle \mathcal{A}(x,\nabla u),\nabla u \rangle dx} = \int_\Omega{\langle \mathcal{B}(x,\mathbf{F}), \nabla u \rangle dx}.
\end{align*}
The conditions on operators $\mathcal{A}$ and $\mathcal{B}$ in~\eqref{eq:dk8} and~\eqref{eq:B-cond} allow us to get
\begin{align}\nonumber
\int_{\Omega} \mathcal{H}(x,\nabla u) dx & \le \tilde{\nu} \int_\Omega{\langle \mathcal{A}(x,\nabla u),\nabla u \rangle dx} = \tilde{\nu}\int_\Omega{\langle \mathcal{B}(x,\mathbf{F}),\nabla u \rangle dx} \\ \nonumber
&\le \tilde{\nu}^2\int_\Omega{|\mathbf{F}|^{p-1}|\nabla u| dx} + C\int_\Omega{|a(x)||\mathbf{F}|^{q-1}|\nabla u| dx}\\ \label{est-I12}
&=: \tilde{\nu}^2(I_1 + I_2).
\end{align}
In order to deal with two terms on the right hand side of~\eqref{est-I12}, we apply H{\"o}lder's and Young's inequalities in Lemma~\ref{lem:HY} for any $\varepsilon>0$. More precisely, the first term $I_1$ can be estimated by applying Lemma~\ref{lem:HY} with 
$$ r = p, \ s = p-1, \ \varphi = |\nabla u|  \ \mbox{ and } \ \psi = |\mathbf{F}|,$$ 
it follows that
\begin{align}\label{est-I12-1}
I_1 \le \varepsilon \int_{\Omega} |\nabla u|^{p}dx + \varepsilon^{-\frac{1}{p-1}} \int_{\Omega} |\mathbf{F}|^{p} dx.
\end{align}
Proceeding analogously, the second term $I_2$ can be estimated with 
$$r = q, \ s = q-1, \ \varphi = |a(x)|^{\frac{1}{q}}|\nabla u| \ \mbox{ and } \ \psi = |a(x)|^{\frac{1}{q}}|\mathbf{F}|,$$ 
to discover that
\begin{align}\nonumber
I_2 & =  \int_\Omega{|a(x)|^{\frac{q-1}{q}}|\mathbf{F}|^{q-1} |a(x)|^{\frac{1}{q}}|\nabla u| dx} \\ \label{est-I12-2}
&\le \varepsilon \int_{\Omega} a(x) |\nabla u|^{q}dx + \varepsilon^{-\frac{1}{q-1}} \int_{\Omega} a(x) |\mathbf{F}|^{q} dx.
\end{align}
Taking~\eqref{est-I12-1} and~\eqref{est-I12-2} into account~\eqref{est-I12}, we may choose $\varepsilon=\frac{1}{2\tilde{\nu}^2}$ to obtain~\eqref{eq:prop1}, which completes the proof of lemma.
\end{proof}

\begin{lemma}\label{lem:global-M-beta}
Let $\beta \in [0,n)$ and $u$ be a distributional solution to~\eqref{eq:main_double}. Then there exists a constant $C=C(\texttt{data},\beta)>0$ such that
\begin{align}\label{eq:M-beta}
\int_{\Omega} \mathbf{M}_{\beta} \left(\mathcal{H}(x,\nabla u)\right) dx \le C  (\mathrm{diam}(\Omega))^{\beta} \int_{\Omega}  \mathcal{H}(x,\mathbf{F}) dx.
\end{align}
\end{lemma}
\begin{proof}
For every $\eta_0>0$, there holds
\begin{align}\nonumber
\int_{\Omega} \mathbf{M}_{\beta} \left(\mathcal{H}(x,\nabla u)\right) dx & = \int_0^{\infty} \mathcal{L}^n \left(\{x \in \Omega: \, \mathbf{M}_{\beta} \left(\mathcal{H}(x,\nabla u)\right) > \eta\}\right) \, d\eta \\ \label{eq:M-beta-1}
& \le C(n) (\mathrm{diam}(\Omega))^n \eta_0 + \int_{\eta_0}^{\infty} \mathcal{L}^n \left(\{x \in \Omega: \, \mathbf{M}_{\beta} \left(\mathcal{H}(x,\nabla u)\right) > \eta\}\right) \, d\eta.
\end{align}
By applying Proposition~\ref{prop1} and the boundedness property of fractional maximal function $\mathbf{M}_{\beta}$ in Lemma~\ref{lem:bound-M-beta}, with $s = 1$, we conclude from~\eqref{eq:M-beta-1} that
\begin{align}\nonumber
\int_{\Omega} \mathbf{M}_{\beta} \left(\mathcal{H}(x,\nabla u)\right) dx 
& \le C(n) (\mathrm{diam}(\Omega))^n \eta_0 + C(n,\beta) \eta_0^{\frac{n}{\beta-n}} \left(\int_{\Omega} \mathcal{H}(x,\nabla u)dx\right)^{\frac{n}{n-\beta}} \\  \label{eq:M-beta-2}
& \le  C(n) (\mathrm{diam}(\Omega))^n \eta_0 + C(\texttt{data},\beta) \eta_0^{\frac{n}{\beta-n}} \left(\int_{\Omega} \mathcal{H}(x,\mathbf{F})dx\right)^{\frac{n}{n-\beta}}.
\end{align}
Let us balance the contribution of two terms on the right hand side of~\eqref{eq:M-beta-2} by taking 
$$\eta_0 = (\mathrm{diam}(\Omega))^{\beta - n} \int_{\Omega} \mathcal{H}(x,\mathbf{F})dx,$$ 
which gives us the desired estimate. 
\end{proof}

\subsection{Local interior estimates}\label{subsec:local}
Let us reproduce a very important result obtained by M. Colombo and G. Mingione in~\cite[Theorem 1.1]{CoMin2016} that may be useful when dealing with problem \eqref{eq:main_double}, as in Lemma \eqref{lem:Mingione} below. This result actually has been extended by S. Byun in \cite[Theorem 2.2]{Byun2017Cava} up to the boundary.
\begin{lemma}
\label{lem:Mingione}
Let $u \in W^{1,1}(\Omega)$ be a distributional solution to \eqref{eq:main_double} under the assumptions \eqref{eq:cond1}, \eqref{eq:cond2} and \eqref{eq:cond3} with $\mathcal{H}(x,Du), \mathcal{H}(x,\mathbf{F}) \in L^1(\Omega)$. Then, for every $\gamma>1$, there exist $r_0=r_0(\texttt{data},\gamma)>0$ and a constant $C=C(\texttt{data},\gamma) \ge 1$ such that the following inequality
\begin{align}
\label{eq:mingione}
\left(\fint_{B_{R/2}}{[\mathcal{H}(x,Du)]^\gamma dx} \right)^{\frac{1}{\gamma}} \le C\fint_{B_R}{\mathcal{H}(x,Du)dx} + C \left(\fint_{B_R}{[\mathcal{H}(x,\mathbf{F})]^\gamma dx} \right)^{\frac{1}{\gamma}},
\end{align}
holds for every ball $B_R \subset \Omega$ such that $R \le r_0$. 
\end{lemma}

We firstly take our attention to the local interior estimates. Let us fix a point $\xi_0 \in \Omega$, $0<R\le r_0/4$ ($r_0$ is defined in Lemma~\ref{lem:Mingione}) and denote $B_{2R}:=B_{2R}(\xi_0)\subset\Omega$. Assume that $u$ is a solution to~\eqref{eq:main_double}, we consider the unique solution $w$ to the following reference problem:
\begin{equation}\label{eq:I1}
\begin{cases} \mbox{div} \left( \mathcal{A}(x,\nabla w)\right) & = \ 0, \quad \ \quad \mbox{ in } B_{2R},\\ 
\hspace{1.2cm} w & = \ u, \qquad \ \mbox{ on } \partial B_{2R}.\end{cases}
\end{equation}
The assertion of Lemma \ref{lem:Mingione} is then applied for solution $w$ to conclude the following local interior estimate \eqref{eq:Rev} stated in the following lemma. We refer the reader to~\cite[Theorem 1.1]{CoMin2016} or~\cite[Proposition 5.3]{Byun2017Cava} for the detail proofs. 

\begin{lemma}\label{lem:Rev}
Let $u \in W^{1,1}(B_{2R})$ be the solution to~\eqref{eq:main_double} satisfying $\mathcal{H}(x,\nabla u) \in L^1(B_{2R})$. Then, there exists a unique distributional solution $w$ to equation \eqref{eq:I1} such that $\mathcal{H}(x,\nabla w) \in L^1(B_{2R})$. Moreover, for every $\gamma>1$ there exists $C = C(\gamma,\texttt{data})>0$ such that
\begin{align}\label{eq:Rev}
\left(\fint_{B_{\rho}(y)} [\mathcal{H}(x,\nabla w)]^{\gamma}dx\right)^{\frac{1}{\gamma}} \le C \fint_{B_{2\rho}(y)} \mathcal{H}(x,\nabla w) dx,
\end{align}
for all $B_{2\rho}(y) \subset B_{2R}$.
\end{lemma}
Next, it enables us to prove the comparison estimates between solution $u$ to problem \eqref{eq:main_double} and $w$ to \eqref{eq:I1} in the interior of domain, via Lemma \ref{lem:comp_est_in} below.
\begin{lemma}\label{lem:comp_est_in}
Let $u \in W^{1,1}(B_{2R})$ be a distributional solution to~\eqref{eq:main_double} under assumptions~\eqref{eq:cond1}, \eqref{eq:cond2}, \eqref{eq:cond3} with $\mathcal{H}(x,\nabla u), \ \mathcal{H}(x,\mathbf{F}) \in L^1(B_{2R})$. Assume that $w \in W^{1,1}(B_{2R})$ is the unique distributional solution to~\eqref{eq:I1} with $\mathcal{H}(x,\nabla w) \in L^1(B_{2R})$. Then there exists a constant $C = C(\texttt{data})>0$ such that
\begin{align}\label{eq:cor}
\fint_{B_{2R}}{\mathcal{H}(x,\nabla u - \nabla w)}dx \le  \varepsilon \fint_{B_{2R}}{\mathcal{H}(x,\nabla u) dx} + C \varepsilon^{-\kappa} \fint_{B_{2R}}{\mathcal{H}(x,\mathbf{F})dx},
\end{align}
for any $\varepsilon \in (0,4^{-q})$, where $\kappa = \max\left\{0,\frac{2-p}{p-1}\right\}$.
\end{lemma}
\begin{proof}
We first remark that $u-w \in W^{1,1}_0(B_{2R})$ and $\mathcal{H}(x,\nabla u - \nabla w) \in L^1(B_{2R})$. Thanks to Lemma~\ref{lem:solu}, it allows us to take $\varphi = u -w$ as a test function in variational formulas of equations~\eqref{eq:main_double} and~\eqref{eq:I1} to deduce that
\begin{align}\label{eq:testf}
\fint_{B_{2R}}{\langle \mathcal{A}(x,\nabla u) - \mathcal{A}(x,\nabla w),\nabla u- \nabla w \rangle dx} = \fint_{B_{2R}}{\langle \mathcal{B}(x,\mathbf{F}),\nabla u- \nabla w  \rangle dx}.
\end{align}
From condition~\eqref{eq:cond3}$_2$ or~\eqref{eq:dk8}, and it follows from~\eqref{eq:testf}, we obtain that
\begin{align}\nonumber
\fint_{B_{2R}} & {\left[(|\nabla u|^2+|\nabla w|^2)^{\frac{p-2}{2}}+a(x)(|\nabla u|^2+|\nabla w|^2)^{\frac{q-2}{2}} \right]|\nabla u-\nabla w|^2 dx} \\ \nonumber
& \qquad \qquad \le  C\fint_{B_{2R}}{\langle \mathcal{A}(x,\nabla u) - \mathcal{A}(x,\nabla w), \nabla u -\nabla w\rangle dx}\\ \nonumber
& \qquad \qquad = C\fint_{B_{2R}}{\langle \mathcal{B}(x,\mathbf{F}),\nabla u - \nabla w \rangle dx} \\ \label{est-2.2}
&\qquad \qquad \le C\fint_{B_{2R}}{\left(|\mathbf{F}|^{p-1}+a(x)|\mathbf{F}|^{q-1} \right)|\nabla u - \nabla w|dx}.
\end{align}
Applying Lemma~\ref{lem:HY} for the right hand side of~\eqref{est-2.2} as in the proof of Proposition~\ref{prop1}, for all $\delta \in (0,1)$ there holds
\begin{align}\nonumber
\fint_{B_{2R}} & \left(|\mathbf{F}|^{p-1}+a(x)|\mathbf{F}|^{q-1} \right)|\nabla u - \nabla w|dx \\  \label{est-2.2b}
& \qquad  \le \delta \fint_{B_{2R}} \mathcal{H}(x,\nabla u - \nabla w)dx + \delta^{-\frac{1}{p-1}} \fint_{B_{2R}} \mathcal{H}(x,\mathbf{F})dx.
\end{align}
By the definition of function $\mathcal{H}$, if $ p \ge 2$ one has
\begin{align*}
\fint_{B_{2R}}&{\mathcal{H}(x,\nabla u-\nabla w)dx}  = \fint_{B_{2R}}{|\nabla u-\nabla w|^p+a(x)|\nabla u-\nabla w|^q  dx}\\
&  \le \fint_{B_{2R}}  {\left[(|\nabla u|^2+|\nabla w|^2)^{\frac{p-2}{2}}+a(x)(|\nabla u|^2+|\nabla w|^2)^{\frac{q-2}{2}} \right]|\nabla u-\nabla w|^2 dx},
\end{align*}
which deduces from~\eqref{est-2.2} and~\eqref{est-2.2b} that
\begin{align*}
\fint_{B_{2R}}{\mathcal{H}(x,\nabla u-\nabla w)dx} \le C\delta \fint_{B_{2R}} \mathcal{H}(x,\nabla u - \nabla w)dx + C \delta^{-\frac{1}{p-1}} \fint_{B_{2R}} \mathcal{H}(x,\mathbf{F})dx.
\end{align*}
We can choose $\delta$ small enough in this inequality to obtain
\begin{align*}
\fint_{B_{2R}}{\mathcal{H}(x,\nabla u - \nabla w)}dx \le C \fint_{B_{2R}}{\mathcal{H}(x,\mathbf{F})dx},
\end{align*}
which obviously implies~\eqref{eq:cor}. In the second case when $1<p<q\le 2$, we can decompose $\mathcal{H}$ as follows
\begin{align}\nonumber
\fint_{B_{2R}}&{\mathcal{H}(x,\nabla u-\nabla w)dx} = \fint_{B_{2R}}{|\nabla u-\nabla w|^p+a(x)|\nabla u-\nabla w|^q  dx}\\ \nonumber
&= \fint_{B_{2R}}{\left(|\nabla u|^2+|\nabla w|^2\right)^{\frac{p(2-p)}{4}}\left[\left(|\nabla u|^2+|\nabla w|^2\right)^{\frac{p(p-2)}{4}}|\nabla u-\nabla w|^p \right] dx} \\ \label{est:H1}
& \quad +\fint_{B_{2R}}{a(x)\left(|\nabla u|^2+|\nabla w|^2\right)^{\frac{q(2-q)}{4}}\left[\left(|\nabla u|^2+|\nabla w|^2\right)^{\frac{q(q-2)}{4}}|\nabla u-\nabla w|^q \right] dx}.
\end{align}
Applying Lemma~\ref{lem:HY} with $r = p$, $s = \frac{p^2}{2}$, 
$$\varphi = \left(|\nabla u|^2+|\nabla w|^2\right)^{\frac{1}{2}}\ \mbox{ and } \ \psi = \left(|\nabla u|^2+|\nabla w|^2\right)^{\frac{p-2}{2p}}|\nabla u-\nabla w|^{\frac{2}{p}},$$ 
for the first term and with $r = q$, $s = \frac{q^2}{2}$, 
$$\varphi = [a(x)]^{\frac{1}{q}}  \left(|\nabla u|^2+|\nabla w|^2\right)^{\frac{1}{2}} \ \mbox{ and } \ \psi = [a(x)]^{\frac{1}{q}}\left(|\nabla u|^2+|\nabla w|^2 \right)^{\frac{q-2}{2q}}|\nabla u-\nabla w|^{\frac{2}{q}},$$  
for the second one on the right hand side of~\eqref{est:H1}, it yields that
\begin{align}\nonumber
 \fint_{B_{2R}}{\mathcal{H}(x,\nabla u-\nabla w)dx}  & \le \varepsilon \fint_{B_{2R}}{\left(|\nabla u|^2+|\nabla w|^2\right)^{\frac{p}{2}} dx} \\ \nonumber
 & \quad + \varepsilon^{1-\frac{2}{p}}\fint_{B_{2R}}{\left(|\nabla u|^2+|\nabla w|^2 \right)^{\frac{p-2}{2}}|\nabla u-\nabla w|^2 dx}\\ \nonumber
 & \quad + \varepsilon \fint_{B_{2R}}{ a(x)\left(|\nabla u|^2+|\nabla w|^2\right)^{\frac{q}{2}} dx} \\  \label{est:H2}
&  \quad  + \varepsilon^{1-\frac{2}{q}} \fint_{B_{2R}} {a(x)\left(|\nabla u|^2+|\nabla w|^2\right)^{\frac{q-2}{2}}|\nabla u-\nabla w|^2  dx}.
\end{align}
for any $\varepsilon \in (0,1)$. By simple computation one has
\begin{align}\nonumber
\fint_{B_{2R}} & \left(|\nabla u|^2+|\nabla w|^2\right)^{\frac{p}{2}}   + a(x)\left(|\nabla u|^2+|\nabla w|^2\right)^{\frac{q}{2}} dx \\ \nonumber
 & \qquad \le 2^q \fint_{B_{2R}}{\left(|\nabla u|+|\nabla u - \nabla w|\right)^{p}  + a(x)\left(|\nabla u|+|\nabla u - \nabla w|\right)^{q} dx} \\  \label{est:H3}
& \qquad \le  2^{2q-1} \fint_{B_{2R}} {\mathcal{H}(x,\nabla u) + \mathcal{H}(x,\nabla u - \nabla w) dx}.
\end{align}
Taking~\eqref{est:H3} into account, it follows from~\eqref{est:H2} that
\begin{align}\nonumber
\fint_{B_{2R}}&{\mathcal{H}(x,\nabla u-\nabla w)dx} 
\le 2^{2q-1} \varepsilon \fint_{B_{2R}}{\mathcal{H}(x,\nabla u)+\mathcal{H}(x,\nabla u - \nabla w)  dx}\\ \nonumber
& + \varepsilon^{1-\frac{2}{p}}\fint_{B_{2R}}{\left[(|\nabla u|^2+|\nabla w|^2)^{\frac{p-2}{2}}+a(x)(|\nabla u|^2+|\nabla w|^2)^{\frac{q-2}{2}} \right]|\nabla u-\nabla w|^2 dx}.
\end{align}
When there is no ambiguity, we may replace $2^{2q}\varepsilon$ by $\varepsilon$ to get that the following inequality
\begin{align}\nonumber
\fint_{B_{2R}}&{\mathcal{H}(x,\nabla u-\nabla w)dx} \le \varepsilon \fint_{B_{2R}}{\mathcal{H}(x,\nabla u)dx}\\  \label{est-2.1}
& + 2\varepsilon^{1-\frac{2}{p}}\fint_{B_{2R}}{\left[(|\nabla u|^2+|\nabla w|^2)^{\frac{p-2}{2}}+a(x)(|\nabla u|^2+|\nabla w|^2)^{\frac{q-2}{2}} \right]|\nabla u-\nabla w|^2 dx},
\end{align}
holds for all $\varepsilon \in (0,4^{-q})$. Combining between~\eqref{est-2.1}, \eqref{est-2.2} and~\eqref{est-2.2b}, for all $\varepsilon \in (0,4^{-q})$ and $\delta \in (0,1)$, there holds
\begin{align}\nonumber
\fint_{B_{2R}}{\mathcal{H}(x,\nabla u-\nabla w)dx} &\le \varepsilon \fint_{B_{2R}}{\mathcal{H}(x,\nabla u)dx} + 2 \varepsilon^{1-\frac{2}{p}} \delta \fint_{B_{2R}}{\mathcal{H}(x,\nabla u- \nabla w)dx}  \\ \label{est-2.3}
& \qquad \qquad + 2 \varepsilon^{1-\frac{2}{p}} \delta^{-\frac{1}{p-1}} \fint_{B_{2R}}{\mathcal{H}(x,\mathbf{F})dx}.
\end{align}
Let us choose $\delta = \frac{1}{4} \varepsilon^{\frac{2}{p}-1}$ in~\eqref{est-2.3}, it results that
\begin{align*}
\fint_{B_{2R}}{\mathcal{H}(x,\nabla u-\nabla w)dx} &\le \varepsilon \fint_{B_{2R}}{\mathcal{H}(x,\nabla u)dx} + C \varepsilon^{-\frac{2-p}{p-1}} \fint_{B_{2R}}{\mathcal{H}(x,\mathbf{F})dx},
\end{align*}
which deduces~\eqref{eq:cor}. Finally, we put the estimate for the case when $1<p<2<q$. To be able to attain the inequality \eqref{eq:cor},we estimate  $\mathcal{H}$ as follows
\begin{align*}
\fint_{B_{2R}} & {\mathcal{H}(x,\nabla u-\nabla w)dx} \\
& \quad  = \fint_{B_{2R}}{\left(|\nabla u|^2+|\nabla w|^2\right)^{\frac{p(2-p)}{4}}\left[\left(|\nabla u|^2+|\nabla w|^2\right)^{\frac{p(p-2)}{4}}|\nabla u-\nabla w|^p \right] dx} \\
& \qquad \quad + \fint_{B_{2R}}  a(x)(|\nabla u|^2+|\nabla w|^2)^{\frac{q-2}{2}} |\nabla u-\nabla w|^2 dx.
\end{align*}
The desired result is implied by using the ones in the above cases. This proves our first assertion interior comparison result.
\end{proof}

\subsection{Local boundary estimates}\label{subsec:boundary}

We now present some comparison results near the boundary of domain. The assertions of following lemmas still hold on the boundary, in a similar way as interior case (see also~\cite{Byun2017Cava}). Let us fix a point $\xi_0 \in \partial \Omega$, $0<R<r_0/4$ and denote $\Omega_{2R}:=B_{2R}(\xi_0)\cap \Omega$. We consider $u$ be a distributional solution to~\eqref{eq:main_double} and $w$ be the unique solution to the corresponding reference problem on $\Omega_{2R}$
\begin{equation}\label{eq:I1-boundary}
\begin{cases} 
\mbox{div} \left( \mathcal{A}(x,\nabla w)\right) & = \ 0, \quad \ \quad \mbox{ in } \Omega_{2R},\\ 
\hspace{1.2cm} w & = \ u, \qquad \ \mbox{ on } \partial \Omega_{2R}.
\end{cases}
\end{equation}

\begin{lemma}\label{lem:Rev-boundary}
Let $u \in W^{1,1}(B_{2R})$ be the solution to~\eqref{eq:main_double} satisfying $\mathcal{H}(x,\nabla u) \in L^1(B_{2R})$. Then, there exists a unique distributional solution $w$ to equation \eqref{eq:I1} such that $\mathcal{H}(x,\nabla w) \in L^1(B_{2R})$. Moreover, for every $\gamma>1$ there exists $C = C(\gamma,\texttt{data})>0$ such that
\begin{align}\label{eq:Rev-boundary}
\left(\fint_{B_{\rho}(y)} [\mathcal{H}(x,\nabla w)]^{\gamma}dx\right)^{\frac{1}{\gamma}} \le C \fint_{B_{2\rho}(y)} \mathcal{H}(x,\nabla w) dx,
\end{align}
for all $B_{2\rho}(y) \subset B_{2R}$.
\end{lemma}

\begin{lemma}\label{lem:comp_boundary}
Let $u \in W^{1,1}(\Omega_{2R})$ be a distributional solution to~\eqref{eq:main_double} under assumptions~\eqref{eq:cond1}, \eqref{eq:cond2} and~\eqref{eq:cond3}. Assume that with $\mathcal{H}(x,\nabla u), \, \mathcal{H}(x,\mathbf{F}) \in L^1(\Omega_{2R})$. Then, $w \in W^{1,1}(\Omega_{2R})$ is a solution to~\eqref{eq:I1-boundary} with $\mathcal{H}(x,\nabla w) \in L^1(\Omega_{2R})$ satisfying
\begin{align}\label{eq:comp_boundary}
\fint_{B_{2R}}{\mathcal{H}(x,\nabla u - \nabla w)}dx \le  \varepsilon \fint_{B_{2R}}{\mathcal{H}(x,\nabla u) dx} + C \varepsilon^{-\kappa} \fint_{B_{2R}}{\mathcal{H}(x,\mathbf{F})dx},
\end{align}
for any $\varepsilon \in (0,4^{-q})$, where $\kappa = \max\left\{0,\frac{2-p}{p-1}\right\}$.
\end{lemma}

\section{Global Lorentz regularity results}
\label{sec:proofs}
By employing comparison estimates in Section \ref{sec:comparison_sec}, we readily prove the main results in this section. It is also worth noting that the good-$\lambda$ technique is a key role behind these proofs. The reader is referred to \cite{55QH4,MPT2018,PNCCM,PNJDE,PNCRM} for this robust technique in previous proofs of regularity results.

\subsection{Good-$\lambda$ theorems of measuring sets}
\label{sec:measure_sets}
As aforementioned, this section constructs some theorems of measuring sets related to the strength of good $\lambda$ method. We start by describing the standard result in measure theory, the corollary of Calder\'on-Zygmund cube decomposition. This covering lemma has many similar versions by independent authors.  The reader is recommended to consult~\cite[Lemma 4.2]{CC1995} for the proof of this lemma. 

\begin{lemma}\label{lem:mainlem}
Let $\Omega$ be a bounded domain of $\mathbb{R}^n$ such that $\partial \Omega \in C^{1,\alpha^+}$. Let $\varepsilon \in (0,1)$ and $R>0$.  Suppose that two measurable subsets $\mathcal{V}\subset \mathcal{W}$ of $\Omega$ satisfy $\mathcal{L}^n\left(\mathcal{V}\right) \le \varepsilon \mathcal{L}^n\left(B_{R}\right)$. Assume moreover that for every $\xi \in \Omega$ and $\varrho \in (0,R]$, we have $B_{\varrho}(\xi) \cap \Omega \subset \mathcal{W}$ if provided $\mathcal{L}^n\left(\mathcal{V} \cap B_{\varrho}(\xi)\right) > \varepsilon \mathcal{L}^n\left(B_{\varrho}(\xi)\right)$. Then there exists a positive constant $C=C(n)$ such that $\mathcal{L}^n\left(\mathcal{V}\right)\leq C \varepsilon \mathcal{L}^n\left(\mathcal{W}\right)$.
\end{lemma} 
\begin{theorem}\label{theo:M_lambda}
Let $\Omega$ be an open bounded domain in $\mathbb{R}^n$ such that $\partial \Omega$ is $C^{1,\alpha^+}$ domain for some $\alpha^+ \in [\alpha,1]$. Assume that $u \in W^{1,1}(\Omega)$ is a distributional solution to~\eqref{eq:main_double} with 
$\mathcal{H}(x,Du), \mathcal{H}(x,\mathbf{F}) \in L^1(\Omega),$
under main assumptions given in~\eqref{eq:cond1}, \eqref{eq:cond2} and~\eqref{eq:cond3}. Then, for any $\vartheta \in (0,1)$, one can find $\varepsilon_0 = \varepsilon_0(n,\vartheta) \in (0,1)$, $\kappa = \kappa(\vartheta) \ge 1$ and a constant $C = C(\texttt{data},\Omega,\vartheta)>0$ such that the following estimate
\begin{align}\nonumber
&\mathcal{L}^n\left(\{{\mathbf{M}}(\mathcal{H}(x,Du))>\varepsilon^{-\vartheta}\lambda, {\mathbf{M}}(\mathcal{H}(x,\mathbf{F})) \le \varepsilon^{\kappa}\lambda \}\cap \Omega \right)\\ \nonumber 
&~~~~~~\qquad \qquad \qquad \qquad \qquad \qquad \leq C \varepsilon \mathcal{L}^n\left(\{ {\mathbf{M}}(\mathcal{H}(x,Du))> \lambda\}\cap \Omega \right),
\end{align}
holds for any $\lambda>0$ and $\varepsilon \in (0,\varepsilon_0)$. 
\end{theorem}
\begin{proof}
We will show that one can find $\varepsilon_0>0$, $\kappa \ge 1$ and a constant $C>0$ such that $\mathcal{L}^n\left(\mathcal{V}_{\varepsilon}\right) \le C \varepsilon \left(\mathcal{W}\right)$ for all $\lambda>0$ and $\varepsilon \in (0,\varepsilon_0)$, where two measurable subsets $\mathcal{V}_{\varepsilon}$, $\mathcal{W}$ of $\Omega$ are defined by 
\begin{align}\nonumber 
&\mathcal{V}_{\varepsilon} = \left\{{\mathbf{M}}(\mathcal{H}(x,\nabla u))>\varepsilon^{-\vartheta}\lambda, {\mathbf{M}}(\mathcal{H}(x,\mathbf{F})) \le \varepsilon^{\kappa}\lambda \right\}\cap \Omega, \\ \nonumber 
& \qquad \qquad \mbox{ and } \mathcal{W} = \left\{ {\mathbf{M}}(\mathcal{H}(x,\nabla u))> \lambda \right\}\cap \Omega.
\end{align}
The main idea is to apply Lemma~\ref{lem:mainlem}. The proof is straightforward when $\mathcal{V}_{\varepsilon} = \emptyset$, so we may assume $\mathcal{V}_{\varepsilon} \neq \emptyset$. Let $R \in (0,r_0/12)$, we first show that $\mathcal{L}^n\left(\mathcal{V}_{\varepsilon}\right) \le \varepsilon \mathcal{L}^n\left(B_{R}\right)$. Indeed, thanks to the fact that the Hardy-Littlewood maximal function ${\mathbf{M}}$ is bounded from Lebesgue space $L^1(\mathbb{R}^n)$ into Marcinkiewicz space $L^{1,\infty}(\mathbb{R}^n)$, one gets that
\begin{align}\label{est-3.1}
\mathcal{L}^n\left(\mathcal{V}_{\varepsilon}\right) \le \mathcal{L}^n\left(\left\{ {\mathbf{M}}(\mathcal{H}(x,\nabla u))>\varepsilon^{-\vartheta}\lambda \right\} \cap \Omega \right) \le \frac{C}{\varepsilon^{-\vartheta} \lambda}\int_{\Omega}{\mathcal{H}(x,\nabla u) dx}.
\end{align}
Using estimate~\eqref{eq:prop1} into \eqref{est-3.1}, we deduce that
\begin{align}\label{est-3.2}
\mathcal{L}^n\left(\mathcal{V}_{\varepsilon}\right) \le \frac{C}{\varepsilon^{-\vartheta} \lambda}\int_{\Omega}{\mathcal{H}(x,\mathbf{F}) dx}.
\end{align}
On the other hand, there exists $\xi_1 \in \Omega$ such that
${\mathbf{M}}(\mathcal{H}(x,\mathbf{F}))(\xi_1) \le \varepsilon^{\kappa}\lambda,$
which follows from \eqref{est-3.2} that
\begin{align*}
\mathcal{L}^n\left(\mathcal{V}_{\varepsilon}\right)  \le \frac{C}{\varepsilon^{-\vartheta} \lambda}\int_{\tilde{B}}{\mathcal{H}(x,\mathbf{F}) dx}  
\le \frac{C \mathcal{L}^n(\Omega)}{\varepsilon^{-\vartheta} \lambda}  {\mathbf{M}}(\mathcal{H}(x,\mathbf{F}))(\xi_1) \le C \varepsilon^{\vartheta+\kappa} \mathcal{L}^n(\Omega),
\end{align*}
where the ball $\tilde{B} : =  B_{D_0}(\xi_1) \supset \Omega$ with $D_0 = \mathrm{diam}(\Omega)$. We remark here the parameter $\kappa$ in this inequality will be determined later such that $\kappa \ge 1$. Therefore we can take $\varepsilon_0$ small enough such that the following estimate 
\begin{align}\nonumber
\mathcal{L}^n\left(\mathcal{V}_{\varepsilon}\right)  \le C \varepsilon^{\vartheta+\kappa} \left(\frac{D_0}{R}\right)^n  \mathcal{L}^n(B_R(\xi_0)) < \varepsilon  \mathcal{L}^n(B_R),
\end{align}
holds for all $\varepsilon \in (0,\varepsilon_0)$. By Lemma~\ref{lem:mainlem}, it remains to prove that for every $\xi \in \Omega$ and $\varrho \in (0,R]$, if $\mathcal{L}^n\left(\mathcal{V}_{\varepsilon} \cap B_{\varrho}(\xi)\right) \ge \varepsilon \mathcal{L}^n\left(B_{\varrho}(\xi)\right)$ then $B_{\varrho}(\xi) \cap \Omega \subset \mathcal{W}$. By contradiction, let us assume that we can find $\xi_2 \in B_{\varrho}(\xi)\cap \Omega \cap \mathcal{W}^c$ and $\xi_3 \in \mathcal{V}_{\varepsilon} \cap B_{\varrho}(\xi)$, this yields
\begin{align}\label{eq:x2}
{\mathbf{M}}(\mathcal{H}(x,\nabla u))(\xi_2) \le \lambda \ \mbox{ and } \
{\mathbf{M}}(\mathcal{H}(x,\mathbf{F}))(\xi_3) \le \varepsilon^{\kappa}\lambda.
\end{align}
The rest of this proof is to show that
\begin{align}\label{eq:iigoal}
\mathcal{L}^n\left(\mathcal{V}_{\varepsilon} \cap B_{\varrho}(\xi)\right) < \varepsilon\mathcal{L}^n\left(B_{\varrho}(\xi)\right).
\end{align}
For any $\zeta \in B_{\varrho}(\xi)$, it is easily to check that
\begin{align}\label{eq:res9}
{\mathbf{M}}(\mathcal{H}(x,\nabla u))(\zeta) \le \max \left\{ {\mathbf{M}}^{\varrho}(\bigchi_{B_{2\varrho}(\xi)} \mathcal{H}(x,\nabla u))(\zeta) ; \  \mathbf{T}^{\varrho} (\mathcal{H}(x,\nabla u))(\zeta) \right\}.
\end{align}
Moreover, since $B_{\varrho'}(\zeta) \subset B_{3\varrho'}(\xi_2)$ for all $\varrho' \ge \varrho$, it follows that
\begin{align}\label{eq:res10}
\mathbf{T}^{\varrho} (\mathcal{H}(x,\nabla u))(\zeta) \le 3^n \sup_{\varrho' > 0}{\fint_{B_{3\varrho'}(\xi_2)}{\mathcal{H}(x,\nabla u) dx}} \le 3^n {\mathbf{M}}(\mathcal{H}(x,\nabla u))(\xi_2).
\end{align}
Combining~\eqref{eq:res9} and~\eqref{eq:res10} with the fact~\eqref{eq:x2}, we get that
\begin{align*}
{\mathbf{M}}(\mathcal{H}(x,\nabla u))(\zeta) \le \max \left\{ {\mathbf{M}}(\bigchi_{B_{2\varrho}(\xi)} \mathcal{H}(x,\nabla u))(\zeta) ; \  3^n \lambda \right\}, \quad \forall \zeta \in B_{\varrho}(\xi).
\end{align*}
Let  $0<\varepsilon_0 < \min\left\{1, 3^{-\frac{n+1}{\vartheta}}\right\}$, then for all $\varepsilon \in (0,\varepsilon_0)$, there holds
\begin{align*}
\mathcal{V}_{\varepsilon} \cap B_{\varrho}(\xi) = \left\{{\mathbf{M}}(\bigchi_{B_{2\varrho}(\xi)} \mathcal{H}(x,\nabla u))> \varepsilon^{-\vartheta}\lambda; \ {\mathbf{M}}(\mathcal{H}(x,\mathbf{F})) \le \varepsilon^{\kappa}\lambda \right\} \cap B_{\varrho}(\xi) \cap \Omega,
\end{align*}
which deduces to
\begin{align}\label{eq:res11}
\mathcal{V}_{\varepsilon} \cap B_{\varrho}(\xi) \subset \left\{{\mathbf{M}}(\bigchi_{B_{2\varrho}(\xi)} \mathcal{H}(x,\nabla u))> \varepsilon^{-\vartheta}\lambda \right\} \cap B_{\varrho}(\xi) \cap \Omega.
\end{align} 

In order to obtain~\eqref{eq:iigoal}, we now consider two cases when $\xi$ belongs to the interior domain $B_{4\varrho}(\xi) \Subset\Omega$ and $\xi$ is close to the boundary $B_{4\varrho}(\xi) \cap \partial\Omega \neq \emptyset$. In the first case $B_{4\varrho}(\xi) \subset \Omega$, let $w$ be the unique solution to the following equation
\begin{equation}\nonumber 
\begin{cases} \mbox{div} \left( \mathcal{A}(x,\nabla w)\right) & = \ 0, \quad \ \quad \mbox{ in } B_{4\varrho}(\xi),\\ 
\hspace{1.2cm} w & = \ u, \qquad \mbox{ on } \partial B_{4\varrho}(\xi).\end{cases}
\end{equation}
Applying Lemma~\ref{lem:Rev} with $\gamma = \frac{1}{\vartheta}>1$, there exists a positive constant $C$ such that
\begin{align}\label{eq:revH-1}
\left(\fint_{B_{2\varrho}(\xi)}{[\mathcal{H}(x,\nabla w)]^{\frac{1}{\vartheta}}} dx\right)^{\vartheta} \le C \fint_{B_{4\varrho}(\xi)}{\mathcal{H}(x,\nabla w) dx} .
\end{align}
Moreover, a comparison estimate between $\nabla u$ and $\nabla w$ over $B_{4\varrho}(\xi)$ can be obtained by~\eqref{eq:cor} in Lemma~\ref{lem:comp_est_in} as follows
\begin{align} \nonumber
\fint_{B_{4\varrho}(\xi)}{\mathcal{H}(x,\nabla u - \nabla w)dx} & \le \varepsilon^{1-\vartheta} \fint_{B_{4\varrho}(\xi)}{\mathcal{H}(x,\nabla u) dx} \\ \nonumber 
& \qquad + C \varepsilon^{-(1-\vartheta)\max\left\{0,\frac{2-p}{p-1}\right\}}\fint_{B_{4\varrho}(\xi)}{\mathcal{H}(x,\mathbf{F})dx} \\ \label{eq:com-1}
& \le \varepsilon^{1-\vartheta} \fint_{B_{4\varrho}(\xi)}{\mathcal{H}(x,\nabla u) dx} + C \varepsilon^{1-\kappa}\fint_{B_{4\varrho}(\xi)}{\mathcal{H}(x,\mathbf{F})dx},
\end{align}
where $\kappa = 1 + \max\left\{0, \frac{(1-\vartheta)(2-p)}{p-1}\right\}$. With the following remark
\begin{align*}
{\mathbf{M}}(\bigchi_{B_{2\varrho}(\xi)} \mathcal{H}(x,\nabla u)) \le C \left[ {\mathbf{M}}(\bigchi_{B_{2\varrho}(\xi)} \mathcal{H}(x,\nabla u - \nabla w)) +  {\mathbf{M}}(\bigchi_{B_{2\varrho}(\xi)} \mathcal{H}(x,\nabla w))\right],
\end{align*}
one deduces from~\eqref{eq:res11} that
\begin{align}\nonumber
\mathcal{L}^n\left(\mathcal{V}_{\varepsilon} \cap B_{\varrho}(\xi)\right) &\le C  \mathcal{L}^n\left(\left\{{\mathbf{M}}(\bigchi_{B_{2\varrho}(\xi)} \mathcal{H}(x,\nabla u - \nabla w))> \varepsilon^{-\vartheta}\lambda \right\} \cap B_{\varrho}(\xi)\right) \\ \label{eq:estV-1}
& \qquad \qquad + C\mathcal{L}^n\left(\left\{{\mathbf{M}}(\bigchi_{B_{2\varrho}(\xi)} \mathcal{H}(x,\nabla w))> \varepsilon^{-\vartheta}\lambda \right\} \cap B_{\varrho}(\xi)\right). 
\end{align}
Applying again the boundedness of the Hardy-Littlewood maximal function ${\mathbf{M}}$ from Lebesgue space $L^s(\mathbb{R}^n)$ into Marcinkiewicz space $L^{s,\infty}(\mathbb{R}^n)$ for two terms on the right hand side of~\eqref{eq:estV-1} with $s = 1$ and $s = \frac{1}{\vartheta}>1$  respectively, we obtain that
\begin{align}\nonumber
\mathcal{L}^n\left(\mathcal{V}_{\varepsilon} \cap B_{\varrho}(\xi)\right) &\le \frac{C}{\varepsilon^{-\vartheta}\lambda} \int_{B_{2\varrho}(\xi)}\mathcal{H}(x,\nabla u - \nabla w) dx \\ \nonumber
& \qquad \qquad + \frac{C}{\left(\varepsilon^{-\vartheta}\lambda\right)^{\frac{1}{\vartheta}}} \int_{B_{2\varrho}(\xi)} \left[\mathcal{H}(x,\nabla w)\right]^{\frac{1}{\vartheta}} dx \\ \nonumber
&\le \frac{C \varrho^n}{\varepsilon^{-\vartheta}\lambda} \fint_{B_{4\varrho}(\xi)}\mathcal{H}(x,\nabla u - \nabla w) dx \\ \label{eq:estV-2}
& \qquad \qquad + \frac{C \varrho^n \varepsilon}{{\lambda}^{\frac{1}{\vartheta}}} \fint_{B_{2\varrho}(\xi)} \left[\mathcal{H}(x,\nabla w)\right]^{\frac{1}{\vartheta}} dx.
\end{align}
It is easy to check that $B_{4\varrho}(\xi) \subset B_{5\varrho}(\xi_2) \cap B_{5\varrho}(\xi_3)$  which follows from~\eqref{eq:x2} that
\begin{align}\label{eq:com-2}
\fint_{B_{4\varrho}(\xi)}{\mathcal{H}(x,\nabla u)dx} \le \left(\frac{5}{4}\right)^n \fint_{B_{5\varrho}(\xi_2)}{\mathcal{H}(x,\nabla u)dx} \le C {\mathbf{M}}(\mathcal{H}(x,\nabla u))(\xi_2) \le C \lambda.
\end{align}
Similarly, thanks to~\eqref{eq:com-1} there holds
\begin{align*}
\fint_{B_{4\varrho}(\xi)}{\mathcal{H}(x,\nabla u - \nabla w)dx} & \le C \varepsilon^{1-\vartheta} \fint_{B_{5\varrho}(\xi_2)}{\mathcal{H}(x,\nabla u) dx} + C \varepsilon^{1-\kappa}\fint_{B_{5\varrho}(\xi_3)}{\mathcal{H}(x,\mathbf{F})dx} \\  
& \le C\varepsilon^{1-\vartheta} {\mathbf{M}}(\mathcal{H}(x,\nabla u))(\xi_2) + C \varepsilon^{1-\kappa}{\mathbf{M}}(\mathcal{H}(x,\mathbf{F}))(\xi_3),
\end{align*}
which implies from~\eqref{eq:x2} that
\begin{align}\label{eq:com-3}
\fint_{B_{4\varrho}(\xi)}{\mathcal{H}(x,\nabla u - \nabla w)dx} & \le  C \left(\varepsilon^{1-\vartheta}+\varepsilon\right)\lambda \le C \varepsilon^{1-\vartheta} \lambda.
\end{align}
Thanks to \eqref{eq:revH-1}, one has
\begin{align}\nonumber
\fint_{B_{2\varrho}(\xi)} \left[\mathcal{H}(x,\nabla w)\right]^{\frac{1}{\vartheta}} dx & \le C \left(\fint_{B_{4\varrho}(\xi)} \mathcal{H}(x,\nabla w) dx\right)^{\frac{1}{\vartheta}} \\ \label{eq:estV-3}
& \le C \left(\fint_{B_{4\varrho}(\xi)} \mathcal{H}(x,\nabla u) + \mathcal{H}(x,\nabla u - \nabla w) dx\right)^{\frac{1}{\vartheta}}.
\end{align}
Adding~\eqref{eq:com-2} and~\eqref{eq:com-3} into~\eqref{eq:estV-3}, it gives
\begin{align*}
\fint_{B_{2\varrho}(\xi)} \left[\mathcal{H}(x,\nabla w)\right]^{\frac{1}{\vartheta}} dx \le C \left(1+\varepsilon^{1-\vartheta}\right)^{\frac{1}{\vartheta}} \lambda^{\frac{1}{\vartheta}} \le C \lambda^{\frac{1}{\vartheta}},
\end{align*}
which deduces~\eqref{eq:iigoal} from~\eqref{eq:com-3} and~\eqref{eq:estV-2}.\\

We finally concentrate to the second case when $\xi$ is close to the boundary of domain $\Omega$, that means $B_{4\varrho}(\xi) \cap \partial\Omega \neq \emptyset$. In this case, we select $\xi_4 \in \partial \Omega$ such that $|\xi_4 - \xi| = \mathrm{dist}(\xi,\partial \Omega) \le 4\varrho$. Let us denote $\tilde{\Omega} = B_{12\varrho}(\xi_4) \cap \Omega$ and consider $\tilde{w}$ as the unique solution to the following equation
\begin{equation}\nonumber 
\begin{cases} \mbox{div} \left( \mathcal{A}(x,\nabla \tilde{w})\right) & = \ 0, \quad \ \quad \mbox{ in } \tilde{\Omega},\\ 
\hspace{1.2cm} \tilde{w} & = \ u, \qquad \mbox{ on } \partial \tilde{\Omega}.\end{cases}
\end{equation}
It notices that $B_{2\varrho}(\xi) \subset B_{6\varrho}(\xi_4)$, \eqref{eq:res11} can be rewritten as
\begin{align*}
\mathcal{V}_{\varepsilon} \cap B_{\varrho}(\xi) \subset \left\{{\mathbf{M}}(\bigchi_{B_{6\varrho}(\xi_4)} \mathcal{H}(x,\nabla u))> \varepsilon^{-\vartheta}\lambda \right\} \cap B_{\varrho}(\xi) \cap \Omega,
\end{align*}
which yields that
\begin{align}\nonumber
\mathcal{L}^n\left(\mathcal{V}_{\varepsilon} \cap B_{\varrho}(\xi)\right) &\le C  \mathcal{L}^n\left(\left\{{\mathbf{M}}(\bigchi_{B_{6\varrho}(\xi_4)} \mathcal{H}(x,\nabla u - \nabla \tilde{w}))> \varepsilon^{-\vartheta}\lambda \right\} \cap B_{\varrho}(\xi)\right) \\ \nonumber
& \qquad \qquad + C\mathcal{L}^n\left(\left\{{\mathbf{M}}(\bigchi_{B_{6\varrho}(\xi_4)} \mathcal{H}(x,\nabla \tilde{w}))> \varepsilon^{-\vartheta}\lambda \right\} \cap B_{\varrho}(\xi)\right) \\ \nonumber
&\le \frac{C \varrho^n}{\varepsilon^{-\vartheta}\lambda} \fint_{B_{12\varrho}(\xi_4)}\mathcal{H}(x,\nabla u - \nabla \tilde{w}) dx \\ \label{eq:est-101}
& \qquad \qquad + \frac{C \varrho^n}{\left(\varepsilon^{-\vartheta}\lambda\right)^{\frac{1}{\vartheta}}} \fint_{B_{6\varrho}(\xi_4)} \left[\mathcal{H}(x,\nabla \tilde{w})\right]^{\frac{1}{\vartheta}} dx.
\end{align}
Thanks to Lemma~\ref{lem:Rev-boundary}, we obtain the reverse H{\"o}lder for the boundary case as follows
\begin{align*}
\left(\fint_{B_{6\varrho}(\xi_4)}{[\mathcal{H}(x,\nabla \tilde{w})]^{\frac{1}{\vartheta}}} dx\right)^{\vartheta} \le C \fint_{B_{12\varrho}(\xi_4)}{\mathcal{H}(x,\nabla \tilde{w}) dx} .
\end{align*}
Application of Lemma~\ref{lem:comp_boundary} enables us to get the following comparison estimate 
\begin{align} \nonumber
\fint_{B_{12\varrho}(\xi_4)}{\mathcal{H}(x,\nabla u - \nabla \tilde{w})dx}  \le \varepsilon^{1-\vartheta} \fint_{B_{12\varrho}(\xi_4)}{\mathcal{H}(x,\nabla u) dx} + C \varepsilon^{1-\kappa}\fint_{B_{12\varrho}(\xi_4)}{\mathcal{H}(x,\mathbf{F})dx}.
\end{align}
Using these inequalities, we can do similar as the proof in the previous case to show these following estimates
\begin{align*}
 \fint_{B_{12\varrho}(\xi_4)}{\mathcal{H}(x,\nabla u)dx} \le C {\mathbf{M}}(\mathcal{H}(x,\nabla u))(\xi_2) \le C \lambda,
\end{align*}
\begin{align*}
\fint_{B_{12\varrho}(\xi_4)}{\mathcal{H}(x,\nabla u - \nabla \tilde{w})dx} & \le  C \left(\varepsilon^{1-\vartheta} {\mathbf{M}}(\mathcal{H}(x,\nabla u))(\xi_2) +\varepsilon^{1-\kappa} {\mathbf{M}}(\mathcal{H}(x,\mathbf{F}))(\xi_3)\right) \\
& \le C \varepsilon^{1-\vartheta} \lambda,
\end{align*}
and
\begin{align*}
\fint_{B_{6\varrho}(\xi_4)} \left[\mathcal{H}(x,\nabla \tilde{w})\right]^{\frac{1}{\vartheta}} dx & \le C \left(\fint_{B_{12\varrho}(\xi_4)} \mathcal{H}(x,\nabla u) + \mathcal{H}(x,\nabla u - \nabla \tilde{w})dx \right)^{\frac{1}{\vartheta}} \\
& \le C \left(1+\varepsilon^{1-\vartheta}\right)^{\frac{1}{\vartheta}} \lambda^{\frac{1}{\vartheta}} \le C \lambda^{\frac{1}{\vartheta}},
\end{align*}
We may conclude that~\eqref{eq:iigoal} also holds in this case by taking into account these inequalities to~\eqref{eq:est-101}. The proof is complete.
\end{proof}

\begin{theorem}\label{theo:M_lambda-beta}
Let $\Omega$ be an open bounded domain in $\mathbb{R}^n$ such that $\partial \Omega$ is $C^{1,\alpha^+}$ domain for some $\alpha^+ \in [\alpha,1]$. Assume that $u \in W^{1,1}(\Omega)$ is a distributional solution to~\eqref{eq:main_double} with 
$\mathcal{H}(x,Du), \mathcal{H}(x,\mathbf{F}) \in L^1(\Omega)$,
under main assumptions given in~\eqref{eq:cond1}, \eqref{eq:cond2} and~\eqref{eq:cond3}. Then for any  $\beta \in [0, n)$ and $\vartheta \in \left(0,1-\frac{\beta}{n}\right)$,  one can find $\varepsilon_0 = \varepsilon_0(n,\beta,\vartheta) \in (0,1)$, $\kappa = \kappa(\beta,\vartheta) \ge 1$ and a constant $C = C(\texttt{data},\Omega,\beta,\vartheta)>0$ such that the following estimate
\begin{align}\nonumber
&\mathcal{L}^n\left(\{{\mathbf{M}}\mathbf{M}_{\beta}(\mathcal{H}(x,Du))>\varepsilon^{-\vartheta}\lambda, {\mathbf{M}_{\beta}}(\mathcal{H}(x,\mathbf{F})) \le \varepsilon^{\kappa}\lambda \}\cap \Omega \right)\\ \nonumber 
&~~~~~~\qquad \qquad \qquad \qquad \qquad \qquad \leq C \varepsilon \mathcal{L}^n\left(\{ {\mathbf{M}}\mathbf{M}_{\beta}(\mathcal{H}(x,Du))> \lambda\}\cap \Omega \right),
\end{align}
holds for any $\lambda>0$ and $\varepsilon \in (0,\varepsilon_0)$. 
\end{theorem}
\begin{proof}
Let $\beta \in [0,n)$ and $u$ be a solution to equations~\eqref{eq:main_double}. For simplicity of notations, we denote 
\begin{align*}
\mathbb{U}_{\beta}(x) = \mathbf{M}_{\beta}(\mathcal{H}(x,\nabla u)) \mbox{ and } \mathbb{F}_{\beta}(x) = \mathbf{M}_{\beta}(\mathcal{H}(x,\mathbf{F})).
\end{align*}
For $\vartheta \in \left(0,1-\frac{\beta}{n}\right)$, we need to prove that we can find $\varepsilon_0 = \varepsilon_0(n,\beta,\vartheta) \in (0,1)$, $\kappa = \kappa(n,\beta,\vartheta) \ge 1$ such that $\mathcal{L}^n \left(\mathcal{V}_{\varepsilon,\beta}\right) \le C \mathcal{L}^n \left(\mathcal{W}_{\beta}\right)$, for all $\lambda>0$ and $\varepsilon \in (0,\varepsilon_0)$, where two measurable sets $\mathcal{V}_{\varepsilon,\beta}$ and $\mathcal{W}_{\beta}$ are defined by
\begin{align} \label{def:VW-beta}
&\mathcal{V}_{\varepsilon,\beta} = \left\{{\mathbf{M}}(\mathbb{U}_{\beta})>\varepsilon^{-\vartheta}\lambda, \, \mathbb{F}_{\beta} \le \varepsilon^{\kappa}\lambda \right\}\cap \Omega \ \mbox{ and } \ \mathcal{W}_{\beta} = \left\{ {\mathbf{M}}(\mathbb{U}_{\beta})> \lambda \right\}\cap \Omega.
\end{align} 
It is similar to the proof of Theorem~\ref{theo:M_lambda}, we may assume that there exists $\xi_1\in \Omega$ such that $\mathbb{F}(\xi_1) \le \varepsilon^{\kappa} \lambda$. Thanks to the boundedness property of maximal function ${\mathbf{M}}$ and notation of $\mathcal{V}_{\varepsilon,\beta}$ in~\eqref{def:VW-beta}, there holds
\begin{align}\label{est-3.1-beta}
\mathcal{L}^n\left(\mathcal{V}_{\varepsilon,\beta}\right) \le \mathcal{L}^n\left(\left\{ {\mathbf{M}}(\mathbb{U}_{\beta})>\varepsilon^{-\vartheta}\lambda \right\} \cap \Omega \right) \le \frac{C}{\varepsilon^{-\vartheta} \lambda}\int_{\Omega}{\mathbb{U}_{\beta}(x) dx}.
\end{align}
Applying estimate~\eqref{eq:M-beta} in Lemma~\ref{lem:global-M-beta} to~\eqref{est-3.1-beta}, we deduce that
\begin{align}\label{est-3.2-beta}
\mathcal{L}^n\left(\mathcal{V}_{\varepsilon,\beta}\right) \le \frac{C D_0^{\beta}}{\varepsilon^{-\vartheta} \lambda}\int_{\Omega}{\mathcal{H}(x,\mathbf{F}) dx} \le \frac{C D_0^n}{\varepsilon^{-\vartheta} \lambda} {\mathbf{M}_{\beta}}(\mathcal{H}(x,\mathbf{F}))(\xi_1),
\end{align}
where $D_0 = \mathrm{diam}(\Omega)$. Let us fix $R \in (0,r_0/18)$. We note that ${\mathbf{M}}_{\beta}(\mathcal{H}(x,\mathbf{F}))(\xi_1) = \mathbb{F}_{\beta}(\xi_1) \le \varepsilon^{\kappa}\lambda$, which follows from~\eqref{est-3.2-beta} that
\begin{align*}
\mathcal{L}^n\left(\mathcal{V}_{\varepsilon,\beta}\right) \le C \varepsilon^{\vartheta+\kappa} \left({D_0}/{R}\right)^n  \mathcal{L}^n(B_R(\xi_0)).
\end{align*}
Since $\kappa \ge 1$, it yields that one can find $\varepsilon_0$ small enough such that $$\mathcal{L}^n\left(\mathcal{V}_{\varepsilon,\beta}\right)  \le \varepsilon  \mathcal{L}^n(B_R), \ \mbox{ for all } \ \varepsilon \in (0,\varepsilon_0).$$ 
As in the proof of Theorem~\ref{theo:M_lambda}, we need to show that for every $\xi \in \Omega$ and $\varrho \in (0,R]$, if $B_{\varrho}(\xi) \cap \Omega \not \subset \mathcal{W}_{\beta}$ then 
\begin{align}\label{eq:iigoal-beta}
\mathcal{L}^n\left(\mathcal{V}_{\varepsilon,\beta} \cap B_{\varrho}(\xi)\right) \le \varepsilon\mathcal{L}^n\left(B_{\varrho}(\xi)\right).
\end{align}
This hypothesis leads to the existence of $\xi_2 \in B_{\varrho}(\xi)\cap \Omega \cap (\mathcal{W}_{\beta})^c$. Moreover, we may assume there is at least $\xi_3 \in \mathcal{V}_{\varepsilon,\beta} \cap B_{\varrho}(\xi)$. Therefore, we have
\begin{align}\label{eq:x2-beta}
{\mathbf{M}}(\mathbb{U}_{\beta})(\xi_2) \le \lambda \ \mbox{ and } \
\mathbb{F}_{\beta}(\xi_3) \le \varepsilon^{\kappa}\lambda.
\end{align}
We now estimate the Lebesgue measure of $\mathcal{V}_{\varepsilon,\beta} \cap B_{\varrho}(\xi)$ via the cut-off fractional maximal function. To do this, we first note that for any $\zeta \in B_{\varrho}(\xi)$, since $B_{\varrho'}(\zeta) \subset B_{3\varrho'}(\xi_2)$ for all $\varrho' \ge \varrho$, one has
\begin{align}\label{eq:res10-beta}
\mathbf{T}^{\varrho} (\mathbb{U}_{\beta})(\zeta) \le 3^n \sup_{\varrho' > 0}{\fint_{B_{3\varrho'}(\xi_2)}{\mathbb{U}_{\beta}(x) dx}} \le 3^n {\mathbf{M}}(\mathbb{U}_{\beta})(\xi_2).
\end{align}
Moreover, for all $\varrho' \in (0,\varrho)$ and $\eta \in B_{\varrho'}(\zeta)$, since $B_{\varrho''}(\eta) \subset B_{\varrho''+3\varrho}(\xi_2)$ for all $\varrho'' \ge \varrho$, it follows that
\begin{align*}
\mathbf{T}_{\beta}^{\varrho}(\mathcal{H}(x,\nabla u))(\eta) & = \sup_{\varrho'' \ge \varrho} (\varrho'')^{\beta - n} \int_{B_{\varrho}(\eta)} \mathcal{H}(x,\nabla u) dx \\
& \le \sup_{\varrho'' \ge \varrho} \left( 1+ \frac{3\varrho}{\varrho''}\right)^{n - \beta} (3\varrho + \varrho'')^{\beta - n} \int_{B_{3\varrho + \varrho''}(\xi_2)} \mathcal{H}(x,\nabla u) dx \\
& \le 4^n \mathbb{U}(\xi_2),
\end{align*}
which yields that
\begin{align}\label{eq:res10b-beta}
\mathbf{M}^r \mathbf{T}_{\beta}^{\varrho}(\mathcal{H}(x,\nabla u))(\zeta) = \sup_{\varrho' \in (0,\varrho)} \fint_{B_{\varrho'}(\zeta)} \mathbf{T}_{\beta}^{\varrho}(\mathcal{H}(x,\nabla u))(\eta) d\eta \le 4^n \mathbf{M}(\mathbb{U})(\xi_2).
\end{align}
Thanks to Lemma~\ref{lem:MrMr}, combining~\eqref{eq:res10-beta} and~\eqref{eq:res10b-beta} with the fact~\eqref{eq:x2-beta}, we get that
\begin{align*}
{\mathbf{M}}(\mathbb{U}_{\beta})(\zeta) 
\le \max \left\{  \mathbf{M}^{2\varrho}_{\beta}(\mathcal{H}(x,\nabla u))(\zeta) ; \  4^n \lambda \right\}, \quad \forall \zeta \in B_{\varrho}(\xi),
\end{align*}
which guarantees that
\begin{align}\label{eq:res11-beta}
\mathcal{L}^n \left(\mathcal{V}_{\varepsilon,\beta} \cap B_{\varrho}(\xi)\right) \le \mathcal{L}^n \left( \left\{{\mathbf{M}}^{2\varrho}_{\beta}(\mathcal{H}(x,\nabla u))> \varepsilon^{-\vartheta}\lambda \right\} \cap B_{\varrho}(\xi) \cap \Omega \right),
\end{align}
for all  $\varepsilon \in (0,\varepsilon_0)$, where $\varepsilon_0^{-\vartheta}>4^n$.\\

To get~\eqref{eq:iigoal-beta}, we now consider two cases when $\xi$ belongs to the interior domain $B_{6\varrho}(\xi) \Subset\Omega$ and $\xi$ is close to the boundary $B_{6\varrho}(\xi) \cap \partial\Omega \neq \emptyset$. In the first case $B_{6\varrho}(\xi) \subset \Omega$, let $w$ be the unique solution to the following equation
\begin{equation}\nonumber 
\begin{cases} \mbox{div} \left( \mathcal{A}(x,\nabla w)\right) & = \ 0, \quad \ \quad \mbox{ in } B_{6\varrho}(\xi),\\ 
\hspace{1.2cm} w & = \ u, \qquad \mbox{ on } \partial B_{6\varrho}(\xi).\end{cases}
\end{equation}
Using the fact that for every $\zeta \in B_{\varrho}(\xi)$, one has $B_{\varrho'}(\zeta) \subset B_{2\varrho}(\zeta)\subset B_{3\varrho}(\xi)$ for all $0 < \varrho'<2\varrho$, we may decompose the cut-off fractional maximal function as follows
\begin{align*}
{\mathbf{M}}^{2\varrho}_{\beta}(\mathcal{H}(x,\nabla u))(\zeta) & = \sup_{0<\varrho'<2\varrho} (\varrho')^{\beta - n} \int_{B_{\varrho'}(\zeta)} \mathcal{H}(x,\nabla u) dx \\ 
& = \sup_{0<\varrho'<2\varrho} (\varrho')^{\beta - n} \int_{B_{\varrho'}(\zeta)} \bigchi_{B_{3\varrho}(\xi)}\mathcal{H}(x,\nabla u) dx \\
& = {\mathbf{M}}^{2\varrho}_{\beta}(\bigchi_{B_{3\varrho}(\xi)} \mathcal{H}(x,\nabla u))(\zeta) \\ & \le C \left[ {\mathbf{M}}^{2\varrho}_{\beta}(\bigchi_{B_{3\varrho}(\xi)} \mathcal{H}(x,\nabla u - \nabla w))(\zeta) +  {\mathbf{M}}^{2\varrho}_{\beta}(\bigchi_{B_{3\varrho}(\xi)} \mathcal{H}(x,\nabla w))(\zeta)\right].
\end{align*}
With this decomposition, one deduces from~\eqref{eq:res11-beta} that
\begin{align}\nonumber
\mathcal{L}^n\left(\mathcal{V}_{\varepsilon,\beta} \cap B_{\varrho}(\xi)\right) &\le C  \mathcal{L}^n\left(\left\{{\mathbf{M}}^{2\varrho}_{\beta}(\bigchi_{B_{3\varrho}(\xi)} \mathcal{H}(x,\nabla u - \nabla w))> \varepsilon^{-\vartheta}\lambda \right\} \cap B_{\varrho}(\xi)\right) \\ \nonumber
& \qquad  + C\mathcal{L}^n\left(\left\{{\mathbf{M}}^{2\varrho}_{\beta}(\bigchi_{B_{3\varrho}(\xi)} \mathcal{H}(x,\nabla w))> \varepsilon^{-\vartheta}\lambda \right\} \cap B_{\varrho}(\xi)\right) \\ \label{eq:estV-1-beta}
& =: \mathrm{I} + \mathrm{II}.
\end{align}
To estimate two terms $\mathrm{I}$ and $\mathrm{II}$ on the right hand side of~\ref{eq:estV-1-beta}, our main idea is using the boundedness of the fractional maximal function ${\mathbf{M}}_{\beta}$ in Lemma~\ref{lem:bound-M-beta} with different values of $s$. The first term $\mathrm{I}$ can be estimated by applying Lemma~\ref{lem:bound-M-beta} with $s = 1$ as follows
\begin{align}\nonumber
\mathrm{I} &\le \frac{C}{\left(\varepsilon^{-\vartheta}\lambda\right)^{\frac{n}{n-\beta}}} \left(\int_{B_{3\varrho}(\xi)}\mathcal{H}(x,\nabla u - \nabla w) dx \right)^{\frac{n}{n-\beta}} \\ \label{eq:estV-2-beta}
&\le \frac{C (6\varrho)^n}{\left(\varepsilon^{-\vartheta}\lambda\right)^{\frac{n}{n-\beta}}} \left((6\varrho)^{\beta} \fint_{B_{6\varrho}(\xi)}\mathcal{H}(x,\nabla u - \nabla w) dx \right)^{\frac{n}{n-\beta}}. 
\end{align}
Applying~\eqref{eq:cor} in Lemma~\ref{lem:comp_est_in}, there holds
\begin{align} \label{eq:com-1-beta}
\fint_{B_{6\varrho}(\xi)}{\mathcal{H}(x,\nabla u - \nabla w)dx} & \le \varepsilon^{1-\vartheta} \fint_{B_{6\varrho}(\xi)}{\mathcal{H}(x,\nabla u) dx} + C \varepsilon^{1-\kappa}\fint_{B_{6\varrho}(\xi)}{\mathcal{H}(x,\mathbf{F})dx},
\end{align}
where $\kappa = 1 + \max\left\{0, \frac{(1-\vartheta)(2-p)}{p-1}\right\}$. We can easily check that 
$$B_{6\varrho}(\xi) \subset B_{7\varrho}(\xi_2) \cap B_{7\varrho}(\xi_3),$$  
which follows from~\eqref{eq:x2-beta} that
\begin{align}\nonumber
(6\varrho)^{\beta}\fint_{B_{6\varrho}(\xi)}{\mathcal{H}(x,\nabla u)dx} & \le \left(\frac{7}{6}\right)^{n - \beta} (7\varrho)^{\beta}\fint_{B_{7\varrho}(\xi_2)}{\mathcal{H}(x,\nabla u)dx} \\ \label{eq:com-2-beta}
& \le \left(\frac{7}{6}\right)^{n - \beta} {\mathbf{M}}(\mathbb{U}_{\beta})(\xi_2) \le \left(\frac{7}{6}\right)^{n - \beta} \lambda,
\end{align}
and 
\begin{align}\label{eq:com-2-beta-F}
(6\varrho)^{\beta}\fint_{B_{6\varrho}(\xi)}{\mathcal{H}(x,\mathbf{F})dx} \le \left(\frac{7}{6}\right)^{n - \beta} \mathbb{F}_{\beta}(\xi_3) \le \left(\frac{7}{6}\right)^{n - \beta} \varepsilon^{\kappa} \lambda.
\end{align}
Combining between~\eqref{eq:com-1-beta}, \eqref{eq:com-2-beta} and~\eqref{eq:com-2-beta-F}, one gets that
\begin{align}\label{eq:com-2b-beta}
(6\varrho)^{\beta} \fint_{B_{6\varrho}(\xi)}{\mathcal{H}(x,\nabla u - \nabla w)dx}  
 & \le C \left( \varepsilon^{1-\vartheta} + \varepsilon\right)\lambda \le C \varepsilon^{1-\vartheta}\lambda,
\end{align}
which follows from~\eqref{eq:estV-2-beta} that
\begin{align}\label{eq:com-3-beta}
\mathrm{I} & \le \frac{C (6\varrho)^n}{\left(\varepsilon^{-\vartheta}\lambda\right)^{\frac{n}{n-\beta}}} \left( \varepsilon^{1-\vartheta}\lambda \right)^{\frac{n}{n-\beta}}  \le {C \varrho^n} \varepsilon^{\frac{n}{n-\beta}} \le C \varepsilon \varrho^n.
\end{align}
In order to estimate the second term $\mathrm{II}$, we apply Lemma~\ref{lem:bound-M-beta} with $s = \frac{1}{\vartheta + \frac{\beta}{n}}>1$, it gives
\begin{align}\nonumber
\mathrm{II} &\le  \frac{C}{\left(\varepsilon^{-\vartheta}\lambda\right)^{\frac{1}{\vartheta + \frac{\beta}{n}}.\frac{n}{n -\frac{\beta}{\vartheta + \frac{\beta}{n}}}}} \left(\int_{B_{3\varrho}(\xi)} \left[\mathcal{H}(x,\nabla w)\right]^{\frac{1}{\vartheta + \frac{\beta}{n}}} dx\right)^{\frac{n}{n -\frac{\beta}{\vartheta + \frac{\beta}{n}}}}.
\end{align}
By directly calculating, it is easy to check that
\begin{align*}
\frac{1}{\vartheta + \frac{\beta}{n}}.\frac{n}{n -\frac{\beta}{\vartheta + \frac{\beta}{n}}} = \frac{1}{\vartheta} \ \mbox{ and } \ \frac{n}{n -\frac{\beta}{\vartheta + \frac{\beta}{n}}} = \frac{1}{\vartheta}\left(\vartheta + \frac{\beta}{n}\right).
\end{align*}
We obtain that
\begin{align}\nonumber
\mathrm{II} &\le  \frac{C}{\left(\varepsilon^{-\vartheta}\lambda\right)^{\frac{1}{\vartheta}}} \left(\int_{B_{3\varrho}(\xi)} \left[\mathcal{H}(x,\nabla w)\right]^{\frac{1}{\vartheta + \frac{\beta}{n}}} dx\right)^{\frac{1}{\vartheta}\left(\vartheta + \frac{\beta}{n}\right)}\\ \nonumber
& =  \frac{C}{\left(\varepsilon^{-\vartheta}\lambda\right)^{\frac{1}{\vartheta}}} \left(\mathcal{L}^n(B_{3\varrho}(x))\fint_{B_{3\varrho}(\xi)} \left[\mathcal{H}(x,\nabla w)\right]^{\frac{1}{\vartheta + \frac{\beta}{n}}} dx\right)^{\frac{1}{\vartheta}\left(\vartheta + \frac{\beta}{n}\right)}\\ \label{eq:estV-5-beta}
& \le   \frac{C (6\varrho)^n}{\left(\varepsilon^{-\vartheta}\lambda\right)^{\frac{1}{\vartheta}}} \left[ (6\varrho)^{\frac{\beta}{\vartheta}} \left(\fint_{B_{3\varrho}(\xi)} \left[\mathcal{H}(x,\nabla w)\right]^{\frac{1}{\vartheta + \frac{\beta}{n}}} dx\right)^{\frac{1}{\vartheta}\left(\vartheta + \frac{\beta}{n}\right)}\right]. 
\end{align}
We now apply Lemma~\ref{lem:Rev} with $\gamma = \frac{1}{\vartheta + \frac{\beta}{n}}>1$, there exists a positive constant $C$ such that
\begin{align}\label{eq:estV-3-beta}
(6\varrho)^{\frac{\beta}{\vartheta}} \left(\fint_{B_{3\varrho}(\xi)} \left[\mathcal{H}(x,\nabla w)\right]^{\frac{1}{\vartheta + \frac{\beta}{n}}} dx\right)^{\frac{1}{\vartheta}\left(\vartheta + \frac{\beta}{n}\right)} & \le C \left((6\varrho)^{\beta}\fint_{B_{6\varrho}(\xi)} \mathcal{H}(x,\nabla w) dx\right)^{\frac{1}{\vartheta}}.
\end{align}
Moreover, thanks to~\eqref{eq:com-2-beta} and~\eqref{eq:com-2b-beta} again, one has
\begin{align}\nonumber
(6\varrho)^{\beta}\fint_{B_{6\varrho}(\xi)} \mathcal{H}(x,\nabla w) dx & \le C  (6\varrho)^{\beta}\fint_{B_{6\varrho}(\xi)} \mathcal{H}(x,\nabla u)dx \\ \nonumber
& \qquad + C (6\varrho)^{\beta}\fint_{B_{6\varrho}(\xi)} \mathcal{H}(x,\nabla u - \nabla w) dx \\ \label{eq:estV-6-beta}
& \le C \left(1 + \varepsilon^{1-\vartheta}\right)\lambda \le C \lambda.
\end{align}
Adding~\eqref{eq:estV-3-beta} and~\eqref{eq:estV-6-beta} into~\eqref{eq:estV-5-beta}, there holds
\begin{align}\label{eq:estV-4-beta}
\mathrm{II} \le \frac{C (6\varrho)^n}{\left(\varepsilon^{-\vartheta}\lambda\right)^{\frac{1}{\vartheta}}} \lambda^{\frac{1}{\vartheta}} \le C \varepsilon \varrho^n.
\end{align}
We may conclude~\eqref{eq:iigoal-beta} from~\eqref{eq:estV-1-beta}, \eqref{eq:com-3-beta} and~\eqref{eq:estV-4-beta}.\\

Finally we prove~\eqref{eq:iigoal-beta} for the second case $B_{6\varrho}(\xi) \cap \partial\Omega \neq \emptyset$. We can take $\xi_4 \in \partial \Omega$ such that $|\xi_4 - \xi| = \mathrm{dist}(\xi,\partial \Omega) \le 6\varrho$. Let us denote $\tilde{\Omega} = B_{18\varrho}(\xi_4) \cap \Omega$ and consider $\tilde{w}$ as the unique solution to the following equation
\begin{equation}\nonumber 
\begin{cases} \mbox{div} \left( \mathcal{A}(x,\nabla \tilde{w})\right) & = \ 0, \quad \ \quad \mbox{ in } \tilde{\Omega},\\ 
\hspace{1.2cm} \tilde{w} & = \ u, \qquad \mbox{ on } \partial \tilde{\Omega}.\end{cases}
\end{equation}
In this case, we remark that for every $\zeta \in B_{\varrho}(\xi)$, one has $B_{\varrho'}(\zeta) \subset B_{2\varrho}(\zeta)\subset B_{9\varrho}(\xi_4)$ for all $0 < \varrho'<2\varrho$, therefore
\begin{align*}
{\mathbf{M}}^{2\varrho}_{\beta}(\mathcal{H}(x,\nabla u))(\zeta) & = \sup_{0<\varrho'<2\varrho} (\varrho')^{\beta - n} \int_{B_{\varrho'}(\zeta)} \mathcal{H}(x,\nabla u) dx \\ 
& = \sup_{0<\varrho'<2\varrho} (\varrho')^{\beta - n} \int_{B_{\varrho'}(\zeta)} \bigchi_{B_{9\varrho}(\xi_4)}\mathcal{H}(x,\nabla u) dx \\
& = {\mathbf{M}}^{2\varrho}_{\beta}(\bigchi_{B_{9\varrho}(\xi_4)} \mathcal{H}(x,\nabla u))(\zeta) \\ & \le C \left[ {\mathbf{M}}^{2\varrho}_{\beta}(\bigchi_{B_{9\varrho}(\xi_4)} \mathcal{H}(x,\nabla u - \nabla \tilde{w}))(\zeta) +  {\mathbf{M}}^{2\varrho}_{\beta}(\bigchi_{B_{9\varrho}(\xi_4)} \mathcal{H}(x,\nabla \tilde{w}))(\zeta)\right].
\end{align*}
With this decomposition, one deduces from~\eqref{eq:res11-beta} that
\begin{align}\nonumber
\mathcal{L}^n\left(\mathcal{V}_{\varepsilon,\beta} \cap B_{\varrho}(\xi)\right) &\le C  \mathcal{L}^n\left(\left\{{\mathbf{M}}^{2\varrho}_{\beta}(\bigchi_{B_{9\varrho}(\xi_4)} \mathcal{H}(x,\nabla u - \nabla \tilde{w}))> \varepsilon^{-\vartheta}\lambda \right\} \cap B_{\varrho}(\xi)\right) \\ \label{eq:estV-1-b-beta}
& \qquad  + C\mathcal{L}^n\left(\left\{{\mathbf{M}}^{2\varrho}_{\beta}(\bigchi_{B_{9\varrho}(\xi_4)} \mathcal{H}(x,\nabla \tilde{w}))> \varepsilon^{-\vartheta}\lambda \right\} \cap B_{\varrho}(\xi)\right).
\end{align}
Applying Lemma~\ref{lem:bound-M-beta} with $s = 1$ and $s = \frac{1}{\vartheta + \frac{\beta}{n}}$ for two terms on the right hand side of~\eqref{eq:estV-1-b-beta} respectively, there holds
\begin{align}\nonumber
\mathcal{L}^n\left(\mathcal{V}_{\varepsilon,\beta} \cap B_{\varrho}(\xi)\right)
&\le \frac{C}{\left(\varepsilon^{-\vartheta}\lambda\right)^{\frac{n}{n-\beta}}} \left(\int_{B_{9\varrho}(\xi_4)}\mathcal{H}(x,\nabla u - \nabla \tilde{w}) dx \right)^{\frac{n}{n-\beta}}\\ \nonumber
& \qquad +  \frac{C}{\left(\varepsilon^{-\vartheta}\lambda\right)^{\frac{1}{\vartheta}}} \left[\left(\int_{B_{9\varrho}(\xi_4)} \left[\mathcal{H}(x,\nabla \tilde{w})\right]^{\frac{1}{\vartheta + \frac{\beta}{n}}} dx\right)^{\frac{1}{\vartheta}\left(\vartheta + \frac{\beta}{n}\right)}\right]\\ \nonumber
 &\le \frac{C (18\varrho)^n}{\left(\varepsilon^{-\vartheta}\lambda\right)^{\frac{n}{n-\beta}}} \left((18\varrho)^{\beta} \fint_{B_{18\varrho}(\xi_4)}\mathcal{H}(x,\nabla u - \nabla \tilde{w}) dx \right)^{\frac{n}{n-\beta}}\\ \label{eq:estV-5-b-beta}
& \qquad +  \frac{C (18\varrho)^n}{\left(\varepsilon^{-\vartheta}\lambda\right)^{\frac{1}{\vartheta}}} \left[ (18\varrho)^{\frac{\beta}{\vartheta}} \left(\fint_{B_{9\varrho}(\xi_4)} \left[\mathcal{H}(x,\nabla \tilde{w})\right]^{\frac{1}{\vartheta + \frac{\beta}{n}}} dx\right)^{\frac{1}{\vartheta}\left(\vartheta + \frac{\beta}{n}\right)}\right]. 
\end{align}
Thanks to Lemma~\ref{lem:Rev-boundary} and Lemma~\ref{lem:comp_boundary}, we also obtain the reverse H{\"o}lder and the comparison estimate as follows
\begin{align*}
\left(\fint_{B_{9\varrho}(\xi_4)}{[\mathcal{H}(x,\nabla \tilde{w})]^{\frac{1}{\vartheta + \frac{\beta}{n}}}} dx\right)^{\vartheta + \frac{\beta}{n}} \le C \fint_{B_{18\varrho}(\xi_4)}{\mathcal{H}(x,\nabla \tilde{w}) dx},
\end{align*}
and 
\begin{align} \nonumber
\fint_{B_{18\varrho}(\xi_4)}{\mathcal{H}(x,\nabla u - \nabla \tilde{w})dx}  \le \varepsilon^{1-\vartheta} \fint_{B_{18\varrho}(\xi_4)}{\mathcal{H}(x,\nabla u) dx} + C \varepsilon^{1-\kappa}\fint_{B_{18\varrho}(\xi_4)}{\mathcal{H}(x,\mathbf{F})dx}.
\end{align}
With the reverse H{\"o}lder's inequality and the above comparison estimate, we obtain from~\eqref{eq:estV-5-b-beta} that
\begin{align*}
\mathcal{L}^n\left(\mathcal{V}_{\varepsilon,\beta} \cap B_{\varrho}(\xi)\right)
 &\le \frac{C (18\varrho)^n}{\left(\varepsilon^{-\vartheta}\lambda\right)^{\frac{n}{n-\beta}}} \left((18\varrho)^{\beta} \fint_{B_{18\varrho}(\xi_4)}\mathcal{H}(x,\nabla u - \nabla \tilde{w}) dx \right)^{\frac{n}{n-\beta}}\\ 
& \qquad +  \frac{C (18\varrho)^n}{\left(\varepsilon^{-\vartheta}\lambda\right)^{\frac{1}{\vartheta}}} \left[\left((18\varrho)^{\beta}\fint_{B_{18\varrho}(\xi_4)} \mathcal{H}(x,\nabla \tilde{w}) dx\right)^{\frac{1}{\vartheta}}\right]. 
\end{align*}
The inequality~\eqref{eq:iigoal-beta} can be established by the similar technique in the comparisons with $\mathbf{M}(\mathbb{U}_{\beta})(\xi_2)$ and $\mathbb{F}_{\beta}(\xi_3)$ as the proof of the previous case. It finishes the proof.
\end{proof}

\subsection{Proofs of main theorems}
\label{sec:main_proofs}
\begin{proof}[Proof of Theorem~\ref{theo:regularityM}]
For $0<t<\infty$ and $0<s<\infty$, let us take $\vartheta \in \left(0,\min\left\{1,\frac{1}{s}\right\}\right)$. Thanks to Theorem~\ref{theo:M_lambda}, there exists $\varepsilon_0>0$ such that the following inequality
\begin{align}\nonumber
&\mathcal{L}^n\left(\{{\mathbf{M}}(\mathcal{H}(x,\nabla u))>\varepsilon^{-\vartheta}\lambda, {\mathbf{M}}(\mathcal{H}(x,\mathbf{F})) \le \varepsilon^{\kappa}\lambda \}\cap \Omega \right)\\ \label{eq:51} 
&~~~~~~\qquad \qquad \qquad \qquad \qquad \qquad \leq C \varepsilon \mathcal{L}^n\left(\{ {\mathbf{M}}(\mathcal{H}(x,\nabla u))> \lambda\}\cap \Omega \right),
\end{align}
holds for any $\lambda>0$ and $\varepsilon \in (0,\varepsilon_0)$. By changing of variables and taking into account~\eqref{eq:51}, it follows that
\begin{align*}
\|{\mathbf{M}}(\mathcal{H}(x,\nabla u))\|^t_{L^{s,t}(\Omega)} & = \varepsilon^{-\vartheta t}s\int_0^\infty{\lambda^t\mathcal{L}^n(\{{\mathbf{M}}(\mathcal{H}(x,\nabla u))>\varepsilon^{-\vartheta}\lambda\} )^{\frac{t}{s}}\frac{d\lambda}{\lambda}}\\
&\le C\varepsilon^{-\vartheta t+\frac{t}{s}}s\int_0^\infty{\lambda^t\mathcal{L}^n\left(\{{\mathbf{M}}(\mathcal{H}(x,\nabla u))>\lambda\}\cap\Omega \right)^{\frac{t}{s}}\frac{d\lambda}{\lambda}}\\
&~~~+ C\varepsilon^{-\vartheta t}s\int_0^\infty{\lambda^t\mathcal{L}^n\left(\{{\mathbf{M}}(\mathcal{H}(x,\mathbf{F}))>\varepsilon^\kappa\lambda\}\cap\Omega \right)^{\frac{t}{s}}\frac{d\lambda}{\lambda}} \\
&\le C\varepsilon^{t\left(\frac{1}{s}-\vartheta\right)}\|{\mathbf{M}}\left(\mathcal{H}(x,\nabla u) \right)\|^t_{L^{s,t}(\Omega)} +C\varepsilon^{-t(\vartheta + \kappa)}\|{\mathbf{M}}(\mathcal{H}(x,\mathbf{F}))\|^t_{L^{s,t}(\Omega)}.
\end{align*}
We may choose $\varepsilon\in (0,\varepsilon_0)$ sufficiently small such that $C\varepsilon^{t\left(\frac{1}{s}-\vartheta\right)} \le \frac{1}{2}$, which completes the proof. The same conclusion can be drawn for the case $t = \infty$.
\end{proof}

\begin{proof}[Proof of Theorem~\ref{theo:main-M-beta}]
Let us set $\beta \in [0,n)$, $t \in (0,\infty)$ and $s \in (0,\infty)$. The same result can be done by the similar method for the case $t = \infty$. We can find $\vartheta \in \left(0,\min\left\{1-\frac{\beta}{n},\frac{1}{s}\right\}\right)$. Thanks to Theorem~\ref{theo:main-M-beta}, there exists $\varepsilon_0>0$ such that the following inequality
\begin{align}\nonumber
&\mathcal{L}^n\left(\{{\mathbf{M}}\mathbf{M}_{\beta}(\mathcal{H}(x,\nabla u))>\varepsilon^{-\vartheta}\lambda, {\mathbf{M}_{\beta}}(\mathcal{H}(x,\mathbf{F})) \le \varepsilon^{\kappa}\lambda \}\cap \Omega \right)\\ \label{eq:51-beta} 
&~~~~~~\qquad \qquad \qquad \qquad \qquad \qquad \leq C \varepsilon \mathcal{L}^n\left(\{ {\mathbf{M}}\mathbf{M}_{\beta}(\mathcal{H}(x,\nabla u))> \lambda\}\cap \Omega \right),
\end{align}
holds for any $\lambda>0$ and $\varepsilon \in (0,\varepsilon_0)$. By changing of variables and taking into account~\eqref{eq:51-beta}, it follows that
\begin{align*}
\|\mathbf{M}{\mathbf{M}}_{\beta}(\mathcal{H}(x,\nabla u))\|^t_{L^{s,t}(\Omega)} & = \varepsilon^{-\vartheta t}s\int_0^\infty{\lambda^t\mathcal{L}^n(\{\mathbf{M}{\mathbf{M}}_{\beta}(\mathcal{H}(x,\nabla u))>\varepsilon^{-\vartheta}\lambda\} )^{\frac{t}{s}}\frac{d\lambda}{\lambda}}\\
&\le C\varepsilon^{-\vartheta t+\frac{t}{s}}s\int_0^\infty{\lambda^t\mathcal{L}^n\left(\{\mathbf{M}{\mathbf{M}}_{\beta}(\mathcal{H}(x,\nabla u))>\lambda\}\cap\Omega \right)^{\frac{t}{s}}\frac{d\lambda}{\lambda}}\\
&~~~+ C\varepsilon^{-\vartheta t}s\int_0^\infty{\lambda^t\mathcal{L}^n\left(\{\mathbf{M}{\mathbf{M}}_{\beta}(\mathcal{H}(x,\mathbf{F}))>\varepsilon^\kappa\lambda\}\cap\Omega \right)^{\frac{t}{s}}\frac{d\lambda}{\lambda}} \\
&\le C\varepsilon^{t\left(\frac{1}{s}-\vartheta\right)}\|\mathbf{M}{\mathbf{M}}_{\beta}\left(\mathcal{H}(x,\nabla u) \right)\|^t_{L^{s,t}(\Omega)} \\
& \qquad \qquad + C\varepsilon^{-t(\vartheta + \kappa)}\|\mathbf{M}{\mathbf{M}}_{\beta}(\mathcal{H}(x,\mathbf{F}))\|^t_{L^{s,t}(\Omega)}.
\end{align*}
Since $t\left(\frac{1}{s}-\vartheta\right)>0$, so we can choose $\varepsilon\in (0,\varepsilon_0)$ satisfying $C\varepsilon^{t\left(\frac{1}{s}-\vartheta\right)} \le \frac{1}{2}$ to obtain that
\begin{align*}
\|\mathbf{M}{\mathbf{M}}_{\beta}(\mathcal{H}(x,\nabla u))\|_{L^{s,t}(\Omega)} \le C\|\mathbf{M}{\mathbf{M}}_{\beta}(\mathcal{H}(x,\mathbf{F}))\|_{L^{s,t}(\Omega)},
\end{align*}
which completes the proof by the boundedness of maximal function. 
\end{proof}


\end{document}